\documentclass[12pt,leqno]{amsart}
\usepackage{amsmath,stmaryrd}
\usepackage{amssymb}
\usepackage{amsthm}
\usepackage{epsf}
\usepackage{graphicx}
\usepackage{esint} 
\usepackage{dsfont}
\usepackage[pagebackref]{hyperref} 
\hypersetup{pdfpagemode=FullScreen, colorlinks=true, citecolor={blue}} 
\usepackage{enumerate} 

\graphicspath{ {./pictures/} }

\numberwithin{equation}{section}

\newcommand{\1}{{\mathds 1}}

\newcommand{\ms}{\medskip}

\newcommand{\R}{\mathbb R}
\newcommand{\N}{\mathbb N}

\newcommand{\bp}{\noindent {\em Proof: }}
\newcommand{\ep}{\hfill $\square$ \medskip}

\newcommand{\A}{\mathcal A}

\newcommand{\C}{\mathcal C}

\theoremstyle{plain}
\newtheorem{theorem}{Theorem}[section]
\newtheorem{lemma}[theorem]{Lemma}
\newtheorem{corollary}[theorem]{Corollary}
\newtheorem{proposition}[theorem]{Proposition}

\newtheorem{definition}[theorem]{Definition}

\theoremstyle{definition}

\theoremstyle{remark}
\newtheorem{remark}[theorem]{Remark}

\begin{document}

\title[Reverse super-Riesz transforms]{Reverse inequalities for super-Riesz transforms on graphs with a slow diffusion}

\author[Feneuil]{Joseph Feneuil}
\address{Joseph Feneuil. Universit\'e Paris-Saclay, Laboratoire de math\'ematiques d’Orsay, 91405, Orsay, France}
\email{joseph.feneuil@universite-paris-saclay.fr}

\begin{abstract} 
In the $D$-dimensional Vicsek graph, we prove that the Riesz-like inequality
\(
\|\nabla f\|_p \leq C \|\Delta^\gamma f\|_p
\)
holds for every $p\in(1,\infty)$ and every
\(
0<\gamma<\gamma^*(p):=\frac{1}{D+1}+\frac{D-1}{D+1}\,\frac{1}{p},
\)
while it fails whenever $p\in(1,\infty)$ and $\gamma^*(p)<\gamma<1$. Thus, the validity of the inequality remains open only at the critical exponent $\gamma=\gamma^*(p)$. This provides, for a non-compact space, an example of  an $L^p$-bounded ``super-Riesz transform'', namely an operator of the form $\nabla \Delta^{-\gamma}$ with $\gamma$ strictly larger than the Euclidean threshold $\frac12$.

To achieve this, we establish a more general result linking the diffusion escape rate and a Poincar\'e inequality on balls to the validity of the reverse Riesz-like inequality
\(
\|\Delta^\gamma f\|_p \leq C \|\nabla f\|_p.
\)

%In the $D$-dimensional Vicsek graph, we establish that the Riesz-like inequality $\||\nabla f|\|_p \leq C \|\Delta^\gamma f \|_p$ holds whenever $p\in (1,\infty)$ and $0 < \gamma < \gamma^*(p):= \frac1{D+1} + \frac1p (\frac{D-1}{D+1})$, and fails whenever $p\in (1,\infty)$ and $ \gamma^*(p) < \gamma < 1$, leaving only the validity of the inequality at critical value $\gamma = \gamma^*(p)$ unknown. We thus gave the first example of $L^p$-boundedness of a ``super Riesz transform'', that is an operator of the form  $\nabla \Delta^{-\gamma}$ where $\gamma$ is greater than the Euclidean value $\frac12$. 
%
%To that objective, we prove a more general result that relates the diffusion escape rate and a Poincar\'e nequality of balls, to the validity of the reverse Riesz-like inequality $ \|\Delta^\gamma f\|_p \leq C  \||\nabla f|\|_p$.
\end{abstract}

\maketitle

\ms\noindent{\bf Keywords:}  Vicsek graphs, graphs, slow diffusion, Poincar\'e inequality, Riesz transform, reverse Riesz transform.

\ms\noindent

\tableofcontents

\section{Introduction and statement of the results}

\subsection{Introduction}
We consider a metric measure space \( (X,d,m) \) endowed with a non-negative self-adjoint operator \( \Delta \) acting on \(L^2(X,m)\). Associated with the corresponding Dirichlet form
\[
\mathcal{E}(f,g) = -\int_X f\,\Delta g\,dm,
\]
one can introduce a notion of gradient operator \(\nabla\), which behaves in \(L^2(X)\) like \(\Delta^{1/2}\). A natural and longstanding question is how these two “differential operators” \(\nabla\) and \(\Delta^{1/2}\) compare in \(L^p\) spaces.

In the Euclidean setting, the Riesz transform \(\mathcal R = \nabla \Delta^{-1/2}\), which can be defined via the Fourier transform or as the singular integral operator
\[
\mathcal{R}f(x)=\mathrm{p.v.}\int_{\mathbb{R}^n}\frac{x-y}{|x-y|^{n+1}}f(y)\,dy,
\]
is bounded on \(L^p(\mathbb{R}^n)\) for all \(p\in (1,\infty)\). One standard proof relies on Calderón–Zygmund theory, where the Riesz transform serves as a model example; details can be found in \cite[Chapter 2]{Ste70a}.

The \(L^p\)-boundedness of the Riesz transform yields the estimate
\begin{equation}\label{Rpbis}\tag{R$_{p}$}
\|\nabla f\|_{L^p} \leq C\,\|\Delta^{1/2}f\|_{L^p}, \qquad \text{for } f\in L^2 \cap L^p.
\end{equation}
From this, one may also consider the reverse estimate
\begin{equation}\label{RRpbis}\tag{RR$_{p}$}
\|\Delta^{1/2}f\|_{L^p} \leq C\,\|\nabla f\|_{L^{p}}.
\end{equation}
Moreover, a simple duality argument shows that \eqref{Rpbis} implies \((\mathrm{RR}_{p'})\), where \(p'\) is the Hölder conjugate of \(p\in (1,\infty)\); see Lemma~\ref{lemDuality} for the precise statement in our setting. As such, in Euclidean spaces, we have the equivalence $\|\Delta^{1/2} f \|_{L^p} \approx \||\nabla f|\|_{L^p}$ for all $p\in (1,\infty)$.

\medskip

In 1983, Strichartz asked in \cite{Str83} for which Riemannian manifolds and which values of \(p\in (1,\infty)\) the inequalities \eqref{Rpbis} and \eqref{RRpbis} hold. This initiated an extensive line of research aimed at identifying necessary and sufficient conditions for these inequalities.

We briefly highlight a few representative results:
\begin{itemize}
\item Under doubling volume growth and Gaussian heat kernel estimates, \eqref{Rpbis} holds for all \(1<p\leq 2\) on manifolds \cite{CD99} and on graphs \cite{Russ00}.
\item Even under the same assumptions as in \cite{CD99}, there exist Riemannian manifolds for which \eqref{Rpbis} fails for all \(p>2\); see \cite[Section 5]{CD99}.
\item If one additionally assumes an \(L^2\)-Poincaré inequality on balls, then \eqref{Rpbis} holds for a small range \(p\in (2,2+\varepsilon)\), see \cite{AC05}.
\item Assuming doubling volume growth and sub-Gaussian heat kernel estimates, \eqref{Rpbis} holds for all \(1<p\leq 2\) on manifolds and graphs \cite{CCFR15}.
\end{itemize}

Further results in the context of Riemannian manifolds or graphs can be found in \cite{CD03,ACDH04,BR09,BF15,DR22,Ouh24}, as well as results using Hardy spaces in the endpoint case \(p=1\)  see \cite{DY,AMR08,HLMMY11,BD14,Fen15}.

\medskip

\begin{figure}[!ht]
\centering
\caption{Fragment of the Vicsek graph of dimension $\log_3 5$.}
\label{fig1}
\includegraphics[width=0.33\textwidth]{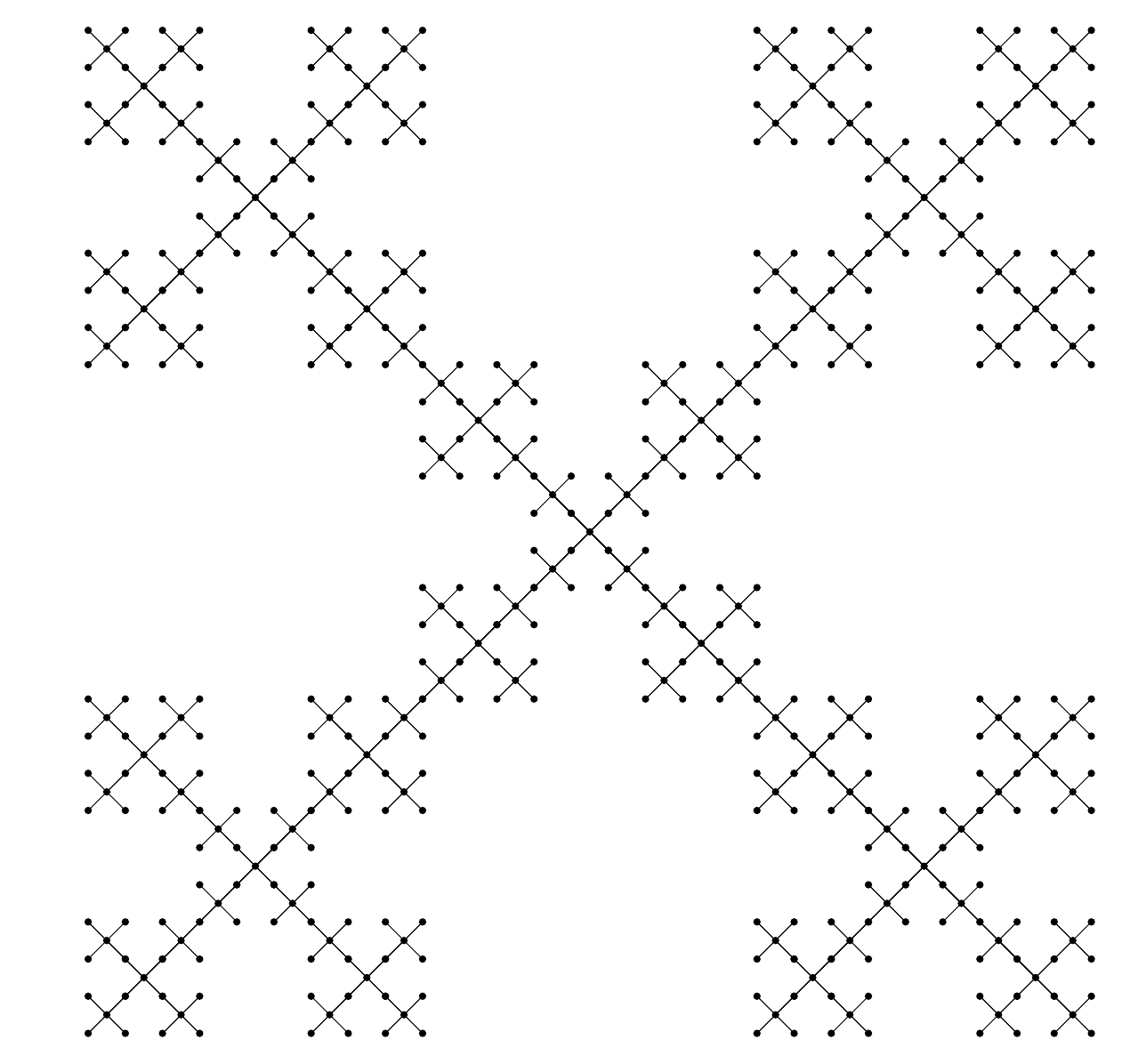}
\end{figure}

The behavior of the Riesz transform in fractal-type structures appears to be strikingly different from the Euclidean case. In \cite{CCFR15}, we established a complete picture of the validity of \eqref{Rpbis} and \eqref{RRpbis} in the case of Vicsek graphs (see Figure~\ref{fig1} for an illustration, and Subsection~\ref{SsVicsek} for the definition). Namely, \eqref{Rpbis} holds if and only if \(p\in (1,2]\), while \eqref{RRpbis} holds if and only if \(p\in [2,\infty)\). Vicsek graphs are particularly well-suited to analysis: their fractal-like structure leads to behavior very different from the Euclidean setting, while their additional structure (in particular their tree-like nature) makes them more tractable than most fractal spaces; see for instance \cite{BC23,BC24}.

This naturally led Devyver and Russ to ask whether, in the case of Vicsek graphs, it is meaningful to compare \(\|\nabla f\|_p\) directly with \(\|\Delta^{1/2} f\|_p\), or whether instead one should compare \(\|\nabla f\|_p\) with \(\|\Delta^{\gamma^*(p)} f\|_p\), for some exponent \(\gamma^*(p)\) to be determined.

\medskip

We therefore introduce a parameter into \eqref{Rpbis} and \eqref{RRpbis}, and study for which \(\gamma \in [0,1]\) and \(p\in (1,\infty)\) the following inequalities hold on Vicsek graphs, or more generally on discrete fractal-type structures:
\begin{equation}\label{Rpggbis}\tag{R$_{p,\gamma}$}
\|\nabla f\|_{L^p} \leq C \|\Delta^\gamma f\|_{L^p}
\quad \text{for any compactly supported function } f,
\end{equation}
and
\begin{equation}\label{RRpggbis}\tag{RR$_{p,\gamma}$}
\|\Delta^\gamma f\|_{L^p} \leq C \|\nabla f\|_{L^p}
\quad \text{for any compactly supported function } f.
\end{equation}

\begin{remark} \label{rmkfinite}
Note that the two inequalities automatically holds when the graph is finite. Indeed, for a graph with $n$ vertices, functions on the graph are vectors in $\R^n$, while the Laplacian is a non-negative symmetric matrix. Therefore, on {\bf finite graphs},  a simple application of the spectral theorem for matrices gives that
\[ \|\Delta^{\gamma_1} f \|_{L^2} \leq  C \|\Delta^{\gamma_2} f \|_{L^2} \quad \text{ whenever } 0 < \gamma_1,\gamma_2.\]
By definition of $\nabla$ and $\Delta$, we also have $\||\nabla f|\|_{L^2} = \|\Delta^{\frac12}\|_{L^2}$. Together with the fact that norms in $\R^n$ are all equivalent, we obtain that 
\[ (C_{p,q,n,\gamma})^{-1} \||\nabla f|\|_{L^p} \leq \|\Delta^{\gamma} f \|_{L^q} \leq C_{p,q,n,\gamma} \||\nabla f|\|_{L^p} \quad \text{whenever $p,q\in [1,\infty]$ and $\gamma>0$,}\]
meaning that \eqref{Rpggbis} and \eqref{RRpggbis} always holds. Computing the optimal constant remains an interesting question, see for instance \cite{ELP}, but will not be the focus of the current article. And in the rest of the presentation, we will always assume that the considered metric space is non-compact.
\end{remark}

In the continuous setting (Riemannian manifolds or cable systems), we instead consider the regularized versions of \eqref{Rpggbis} and \eqref{RRpggbis}:
\begin{equation}\label{Rpgmbis}\tag{R$_{p,\gamma}$}
\|\nabla e^{-\Delta} f\|_{L^p} \leq C \|\Delta^\gamma f\|_{L^p},
\end{equation}
and
\begin{equation}\label{RRpgmbis}\tag{RR$_{p,\gamma}$}
\|\Delta^\gamma e^{-\Delta} f\|_{L^p} \leq C \|\nabla f\|_{L^p},
\end{equation}
for \(f \in \mathcal D_{L^2}(\Delta)\), the domain of \(\Delta\) in \(L^2\).
The distinction between discrete and continuous settings is essential. On graphs, both the Laplacian and the gradient are bounded on \(L^p\), and it is therefore meaningful to compare \(\nabla\) with arbitrary powers of \(\Delta\). In contrast, on Riemannian manifolds or cable systems, the local structure is Euclidean, and a heat semigroup regularization is needed to isolate the genuinely large-scale, fractal behavior.

The inequalities \eqref{Rpggbis} and \eqref{RRpggbis}, referred to as quasi-Riesz inequalities and reverse quasi-Riesz inequalities, were first introduced in \cite{Chen14}. Positive and negative results in fractal-type settings were subsequently obtained in \cite{DRY}, \cite{DR24} and \cite{Fen26}. A first indication that the exponent \(1/2\) is not the correct scaling for the gradient in the sub-Gaussian regime was obtained by Devyver and Russ in \cite{DR24}. For Vicsek cable systems of dimension \(D\), they showed that the reverse inequality is governed by the critical exponent
\[
\gamma^*(p):=\frac1{D+1}+\frac1p\left(1-\frac{2}{D+1}\right),
\]
rather than \(1/2\). More precisely, they proved the following result.

\begin{theorem}[{\cite[Theorem 1.8]{DR24}}] \label{ThDR}
Let \(G\) be a Vicsek cable system satisfying the Ahlfors regularity condition
\[
V(x,r)\simeq r^{D}, \qquad r\geq 1.
\]
Then:
\begin{enumerate}[(i)]
\item the estimate \eqref{RRpgmbis} holds for \(p\geq 2\) whenever \(\gamma \geq \tfrac12\), and for \(1<p<2\) whenever
\[
\gamma > \gamma^*(p):= \frac1{D+1}+\frac1p\left(1-\frac{2}{D+1}\right);
\]
\item the estimate \eqref{RRpgmbis} fails whenever
\[
\gamma < \gamma^*(p).
\]
\end{enumerate}
By duality, this further implies the failure of \eqref{Rpgmbis} for all \(p\in (1,\infty)\) and \(\gamma > \gamma^*(p)\).
\end{theorem}

\begin{figure}[!ht]
\centering
\caption{Theorem~\ref{ThDR}: validity region for \eqref{RRpgmbis} on Vicsek cable systems.}
\label{figRpV}
\includegraphics[width=0.5\textwidth]{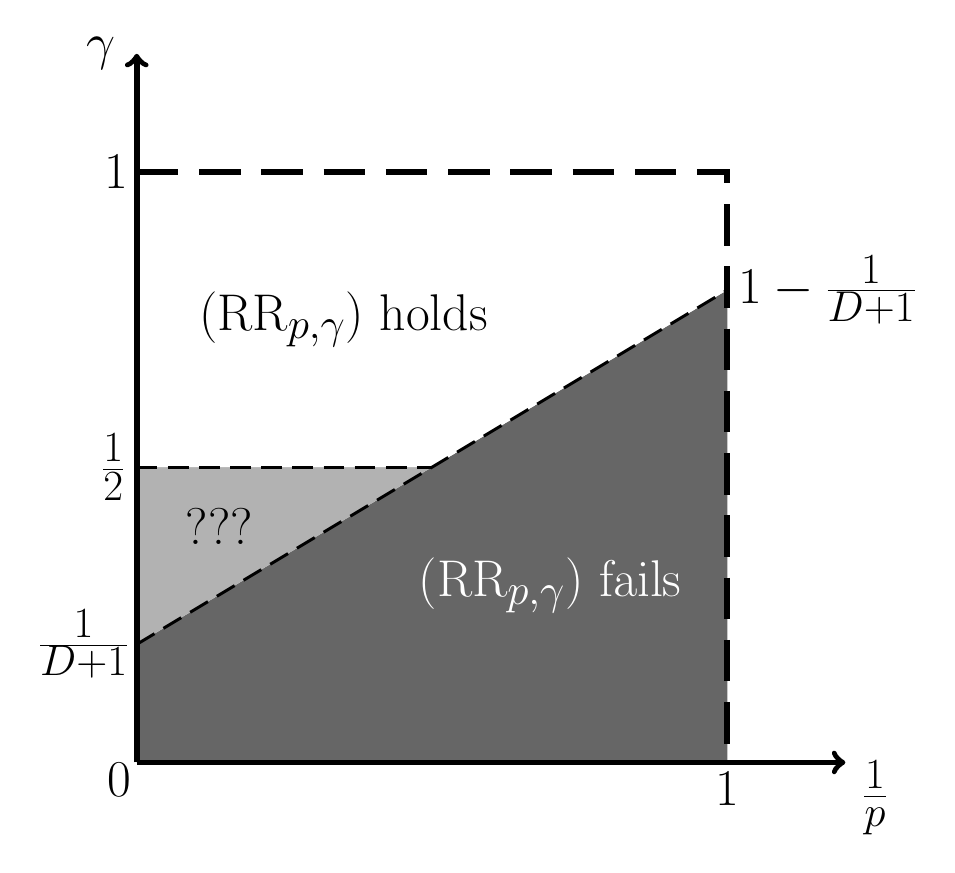}
\end{figure}

Devyver and Russ conjectured that \eqref{RRpgmbis} should also hold in the grey region in Figure~\ref{figRpV}, namely when \(p\in (2,\infty)\) and \(\gamma^*(p) < \gamma < \tfrac12\). This range is particularly challenging, since it corresponds to the regime where \(\gamma\) drops below the Euclidean threshold \(1/2\). The main goal of this article is to prove their conjecture in the setting of Vicsek graphs, which are the discrete analogue of Vicsek cable systems.

We also study the validity of \eqref{Rpgmbis} outside the critical value \(\gamma=\gamma^*(p)\), thereby establishing for the first time\footnote{In a non-compact space, see Remark \ref{rmkfinite}.} \(L^p\)-boundedness results for {\bf super-Riesz transforms}, namely operators of the form \(\nabla \Delta^{-\gamma}\) with \(\gamma>1/2\).

To achieve this, we establish a sufficient condition, expressed in terms of Poincar\'e inequalities, ensuring the validity of \eqref{RRpggbis}. We then exploit the tree structure of the Vicsek graph to show that Sobolev spaces (and hence Poincar\'e inequalities) interpolate, and that \eqref{RRpggbis} implies (R$_{p',\gamma}$), leading to the desired result.

\subsection{Main results}

Our first result completes the picture presented in Figure~\ref{fig1} concerning the Riesz and reverse Riesz-type inequalities on Vicsek graphs.

\begin{theorem} \label{ThVicsek}
Let $(G,\mu)$ be a $D$-dimensional Vicsek graph. Then, for any $p\in (1,\infty)$ and any $\gamma$ satisfying
\[
0 < \gamma < \gamma^*(p):= \frac1{D+1} + \frac1p \Big(1-\frac{2}{D+1}\Big),
\]
the Riesz transform inequality \eqref{Rpggbis} holds and the reverse Riesz transform inequality \eqref{RRpggbis} fails.

\medskip

By duality, for any $p\in (1,\infty)$ and any $\gamma$ satisfying
\[
\gamma^*(p) < \gamma < 1,
\]
the Riesz transform inequality \eqref{Rpggbis} fails and the reverse Riesz transform inequality \eqref{RRpggbis} holds.
\end{theorem}

\begin{figure}[!ht]
\centering
\caption{Theorem~\ref{ThVicsek}: validity region for \eqref{Rpggbis} and \eqref{RRpggbis} on Vicsek graphs.}
\label{figRpVV}
\includegraphics[width=0.95\textwidth]{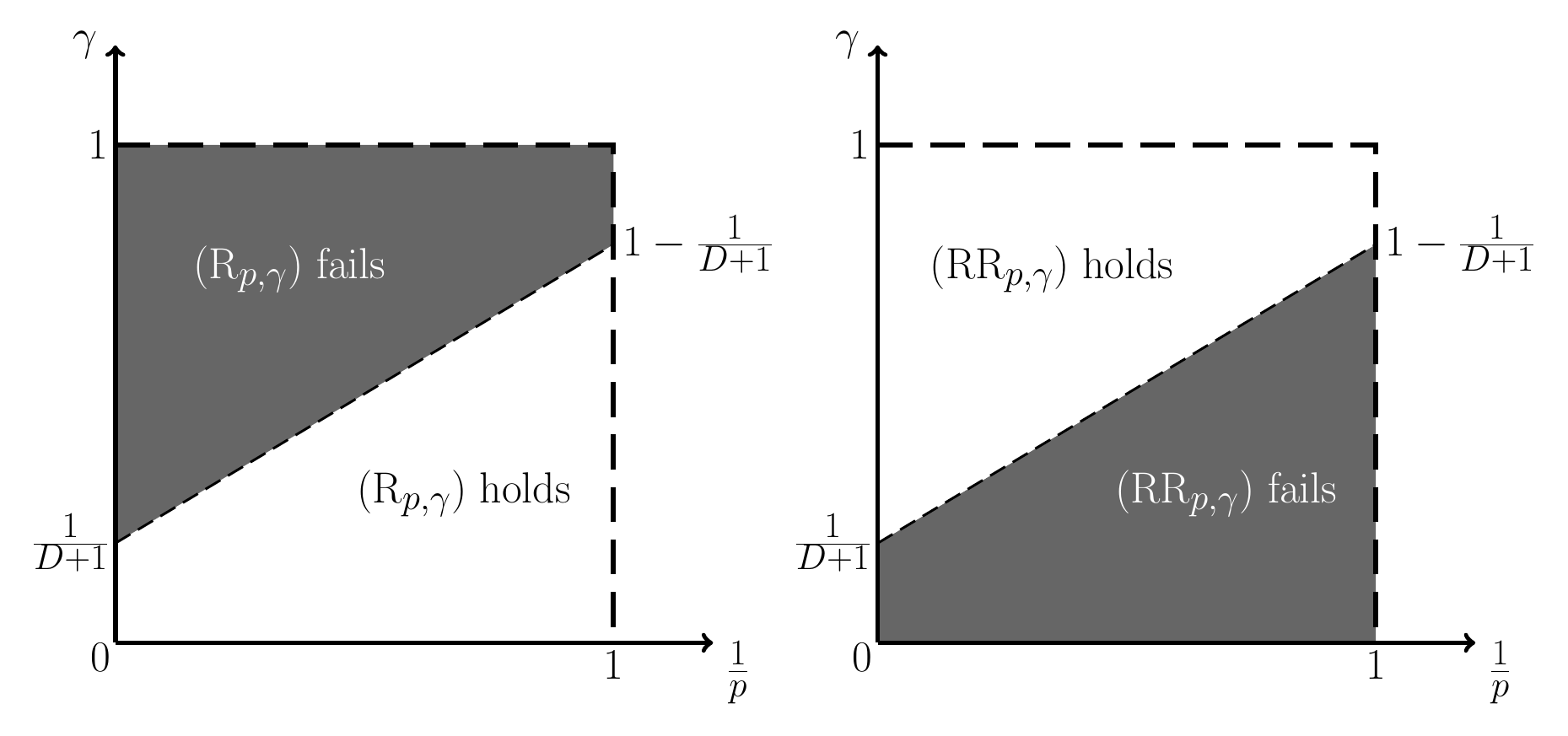}
\end{figure}

To establish this result on Vicsek graphs, we analyze the interplay between \(L^q\) Poincar\'e inequalities on balls, pointwise sub-Gaussian estimates for the diffusion semigroup kernel, and reverse Riesz-type inequalities. More precisely, we relate the growth rate of the Poincar\'e constant on balls of radius \(r\), the escape rate of the associated diffusion, and the power of \(\Delta\) that can be compared with the gradient. This strategy is not entirely surprising, and in fact a related approach already appears in the work of Devyver and Russ \cite{DR24}.

\medskip

Our second main result can be stated as follows.  The precise notion of graphs and of our assumptions can be found in Subsections~\ref{defgraphs} and \ref{Sscondgraphs}.

\begin{theorem} \label{ThMain}
Let $(G,\mu)$ be a (weighted, unoriented, infinite, connected) graph satisfying:
\begin{itemize}
\item the condition \eqref{LB}, which yields analyticity of the associated diffusion and a uniform lower bound on the transition density;
\item the doubling property of balls \eqref{D};
\item a sub-Gaussian pointwise estimate for the random walk \ref{UE}, where $\beta$ is such that the average  escape time from balls of radius $r$ is of order $r^\beta$;
\item an $L^q$-Poincar\'e inequality on balls \ref{Pq}, with constant growing like $r^s$ for balls of radius $r$.
\end{itemize}
Then, for any $p\in (q,\infty)$ and any $\gamma \in (\tfrac{s}{\beta},1)$, the estimate \eqref{RRpggbis} holds; that is, there exists $C>0$ such that
\[
\|\Delta^\gamma f\|_{L^p(G)} \leq C \|\nabla f\|_{L^p(G)}
\quad \text{for all compactly supported functions } f \text{ on } G.
\]
\end{theorem}

The previous theorem yields the following consequence.

\begin{corollary} \label{CorMain}
Let $(G,\mu)$ be a (weighted, unoriented, infinite, connected) graph satisfying:
\begin{itemize}
\item the condition \eqref{LB}, which yields analyticity of the associated diffusion and a uniform lower bound on the transition density;
\item the Ahlfors regularity condition \ref{AR}, where $D$ is the Ahlfors dimension;
\item a sub-Gaussian pointwise estimate for the random walk \ref{UE}, where $\beta$ is such that the average escape time from balls of radius $r$ is of order $r^\beta$.
\end{itemize}
If $D < \beta$, then for any $p\in (1,\infty)$ and any $\gamma \in (\tfrac{D}{\beta},1)$, the reverse quasi-Riesz transform estimate \eqref{RRpggbis} holds.
\end{corollary}

This corollary can be related to \cite[Theorem 1.4]{DRY}. In that work, the authors establish a variant of \eqref{Rpgmbis} for certain cable systems, assuming—in addition to conditions analogous to ours—a reverse H\"older estimate for the gradient of solutions to $\Delta u = 0$. Ignoring the difference in setting, the assumptions in \cite{DRY} are stronger than ours, but they obtain a stronger result than ours: the validity of \eqref{Rpgmbis} for $p\in (1,\infty)$ and $\gamma \in (0,1-\tfrac{D}{\beta})$. Indeed, note that by duality \eqref{Rpgmbis} for $p\in (1,\infty)$ and $\gamma \in (0,1-\tfrac{D}{\beta})$ yields \eqref{RRpgmbis} in the same range of $p$ and $\gamma$ as in Corollary~\ref{CorMain}.

\subsection{Plan of the article and notation}

In Section~\ref{S2}, we introduce the setting of weighted, unoriented, infinite, connected graphs equipped with a diffusion and a gradient. We then present our assumptions and some immediate consequences, and conclude by defining Vicsek graphs. In Section~\ref{S3}, we develop the theory of tent spaces and Lusin-type functionals needed in the sequel. Section~\ref{S4} is devoted to the proof of Theorem~\ref{ThMain} and Corollary~\ref{CorMain}, while Section~\ref{S5} focuses on Vicsek graphs: we prove complex interpolation results for Sobolev spaces, and combine them with the results of Section~\ref{S4} to establish Theorem~\ref{ThVicsek}.

\medskip

Throughout the paper, we write $A \lesssim B$ if $A \leq C B$ for some constant $C$ whose dependence on parameters is either specified or clear from context. We write $A \approx B$ if $A \lesssim B$ and $B \lesssim A$.

\subsection*{Acknowledgments} The author would like to thank Cédric Arhancet for his valuable comments on an earlier version of this manuscript, particularly for drawing attention to the existence of super-Riesz transforms on finite graphs (see Remark~\ref{rmkfinite}).

\section{The discrete setting} \label{S2}

\subsection{Definitions of gradient and Laplacian on graphs} \label{defgraphs}

Let \(G\) be an infinite set and let \(\mu_{xy}=\mu_{yx}\ge 0\) be a symmetric weight on \(G\times G\).
The pair \((G,\mu)\) induces a (weighted, unoriented) graph structure by defining the set of edges
\[
E_G := \{(x,y)\in G\times G \,:\, \mu_{xy}>0\}.
\]
We say that \(x\) and \(y\) are neighbors, and write \(x\sim y\), if and only if \((x,y)\in E_G\).
The \textbf{distance} \(\delta(x,y)\) between two vertices \(x,y\in G\) is the length of the shortest path connecting them, i.e.
\begin{multline*}
\delta(x,y) := \min\Big\{N\in \mathbb{N} \,:\, \exists (x_0,\dots,x_N)\in G^{N+1} \text{ such that } \\
x_0=x,\; x_N=y,\; \text{and } (x_{i-1},x_i)\in E_G \text{ for all } i=1,\dots,N\Big\}.
\end{multline*}
We always assume that \(G\) is \textbf{connected}, meaning that \(\delta(x,y)<\infty\) for all \(x,y\in G\).

\smallskip

We define the {\bf weight} \(m(x)\) of a vertex \(x\in G\) by
\[
m(x) := \sum_{y\sim x} \mu_{xy}.
\]
More generally, the measure of a subset \(E\subset G\) is
\[
m(E):=\sum_{x\in E} m(x).
\]
For \(1\le p<\infty\), we say that a function \(f\) on \(G\) belongs to \(L^p(G,m)\) (or simply \(L^p(G)\)) if
\[
\|f\|_{L^p(G)} := \left(\sum_{x\in G} |f(x)|^p m(x)\right)^{1/p}<\infty,
\]
while \(L^\infty(G)\) consists of functions such that
\[
\|f\|_{L^\infty(G)} := \sup_{x\in G} |f(x)| < \infty.
\]

\smallskip

For \(x\in G\) and \(r>0\), we define the \textbf{ball}
\[
B(x,r):=\{y\in G \,:\, \delta(x,y)<r\}.
\]
We say that \(B\subset G\) is a ball if there exist \(x_B\in G\) and \(r_B>0\) such that \(B=B(x_B,r_B)\).
Note that the center and radius are not necessarily unique, but this will play no role in what follows.
As usual, for a ball \(B\) or \((x,r)\in G\times(0,\infty)\), we set
\[
V(B):=m(B), \qquad V(x,r):=V(B(x,r)).
\]

\smallskip

We also equip \(E_G\) with the measure \(\mu(e):=\mu_{xy}\) for \(e=(x,y)\in E_G\).
We say that \(\omega\) is a \textbf{1-form} on \(G\) if \(\omega\) is a function on \(E_G\) such that
\[
\omega(x,y)=-\omega(y,x)\qquad \text{for all } (x,y)\in E_G,
\]
in particular \(\omega(x,x)=0\).
We define
\[
|\omega|(x):=\left(\frac{1}{2m(x)}\sum_{y\sim x}\mu_{xy}|\omega(x,y)|^2\right)^{1/2}, \qquad x\in G,
\]
which allows us to view 1-forms as functions when estimating them.
For instance, we sometimes write \( |\omega|\in L^p(G)\) instead of \(\omega\in L^p(E_G,\mu)\).
Note that when \(\mu_{xy}\simeq 1\) (as in Vicsek graphs introduced later), we have
\[
|\omega|\in L^p(G) \quad \Longleftrightarrow \quad \omega\in L^p(E_G,\mu).
\]

\medskip

The \textbf{differential operators} \(\nabla\) and \(\nabla^*\) are defined by
\[
\nabla f(x,y):=f(x)-f(y), \qquad (x,y)\in E_G,
\]
for functions \(f\) on \(G\), and
\[
\nabla^*\omega(x):=\frac{1}{m(x)}\sum_{y\sim x}\mu_{xy}\,\omega(x,y), \qquad x\in G,
\]
for 1-forms \(\omega\).
One can check the duality identity
\begin{equation} \label{dd*duals}
\frac12\sum_{e\in E_G} \nabla f(e)\,\omega(e)\,\mu(e)
= \sum_{x\in G} f(x)\,\nabla^*\omega(x)\,m(x),
\end{equation}
whenever all quantities are well-defined.

\medskip

The (nonnegative) Laplacian is defined by
\begin{equation} \label{defDelta}
\Delta := \nabla^*\nabla.
\end{equation}
For \(x,y\in G\), we define the reversible Markov kernel
\[
p(x,y):=\frac{\mu_{xy}}{m(x)m(y)}.
\]
For functions \(f\) on \(G\), we define
\begin{equation}\label{Fen4defP}
Pf(x)=\frac{1}{m(x)}\sum_{y\sim x}\mu_{xy}f(y).
\end{equation}
One verifies that
\begin{equation} \label{Delta=I-P}
\Delta = I-P.
\end{equation}

The discrete kernel \(p_k(x,y)\) is defined by
\[
p_0(x,y)=\frac{\delta(x,y)}{m(y)}, \qquad
p_{k+1}(x,y)=\sum_{z\in G} p(x,z)p_k(z,y)m(z).
\]
Moreover,
\begin{equation} \label{sumpk=1}
\sum_{y\in G} p_k(x,y)m(y)=1,
\qquad
p_k(x,y)=p_k(y,x).
\end{equation}

\medskip

Finally, whenever \(|\nabla f|\in L^2(G)\),
\begin{multline}
\|\nabla f\|_{L^2(G)}^2
= \sum_{x\in G}\frac{1}{2m(x)}\sum_{y\sim x}\mu_{xy}|f(x)-f(y)|^2\,m(x) \\
= \frac12\sum_{e\in E_G}\nabla f(e)^2\,\mu(e)
= \sum_{x\in G} f(x)\Delta f(x)\,m(x)
= \|\Delta^{1/2}f\|_{L^2(G)}^2.
\end{multline}

\medskip

We finish the paragraph with the following duality result.

\begin{lemma} \label{lemDuality}
Let $(G,\mu)$ be a graphs. Then, for any $p\in (1,\infty)$ and any $\gamma \in [0,1]$, we heve
\[ \text{\eqref{Rpggbis} $\implies$ {\normalfont (RR$_{p',1-\gamma}$)}},\]
where $p'$ is the H\"older conjugate of $p$.
\end{lemma}

\bp
This result is classical, and analogue versions have already been proven in other settings, like for instance in cable systems (see \cite[Lemma 1.4]{DR24}). We prove it for completeness.

\medskip

\noindent {\bf (a) $\Delta^\gamma$ has dense range in $L^p$.} Since $\gamma \in [0,1]$ and $\Delta$ is bounded, $\Delta^\gamma$ has dense range if $\Delta$ has also dense range. Let us prove the later.

If by contradiction, $\Delta$ does not have a dense range in $L^p$, then by the Hahn-Banach theorem, there exists $g \in L^{p'}(G) \setminus \{0\}$ such that 
\[ \left< g, \Delta f \right> = \sum_{z\in G} g(z) \Delta f(z) m(z) = 0 \qquad \text{ for all } f\in L^p.\]
So by the symmetry of ($P$ and) $\Delta$
\[ \left< \Delta g, f \right>  = \left<g,  \Delta f \right>  = 0 \qquad \text{ for all } f\in L^p,\]
meaning that $\Delta g = 0$. Since $g\in L^{p'}(G)$, it has a global maximum, and if $x$ is such that $g(x) = \max_G g$, we have 
\[0 = \Delta g(x) = \sum_{y\sim x} \underbrace{[g(x) - g(y)]}_{\geq 0} p(x,y) m(y), \]
meaning that $g(y) = g(x) = \max_G g$ for all the neighbors $y$ of $x$. Since $G$ is connected, it means that $g$ is constant equal to $\max_G g$, and since $g\in L^{p'}$, it means that $g\equiv 0$. A contradiction with our choice of $g$.

\medskip

\noindent {\bf (b) Duality.} Since $\Delta^\gamma$ has dense range, 
\begin{multline*}
 \|\Delta^{1-\gamma} f\|_{L^{p'}(G)} = \sup_{g \in L^p(G) \atop \|\Delta^\gamma g\|_p \leq 1} \sum_{z\in G} \Delta^{1-\gamma} f(z) \, \Delta^\gamma g(z) m(z) =  \sup_{g \in L^p(G) \atop \|\Delta^\gamma g\|_p \leq 1} \sum_{z\in G} [\nabla^* \nabla f](z) \,  g(z) m(z) \\
 =  \sup_{g \in L^p(G) \atop \|\Delta^\gamma g\|_p \leq 1} \sum_{e\in E_G} \nabla f(e) \, \nabla g(e) \mu(e)
 \end{multline*}
 by \eqref{dd*duals}. We then continue as
\begin{multline*}
 \|\Delta^{1-\gamma} f\|_{L^{p'}(G)} \leq \sup_{g \in L^p(G) \atop \|\Delta^\gamma g\|_p \leq 1} \sum_{z\in G} |\nabla f|(z) \, |\nabla g|(z) m(z) \leq \sup_{g \in L^p(G) \atop \|\Delta^\gamma g\|_p \leq 1} \| |\nabla f|\|_{L^{p'}(G)} \||\nabla g|\|_{L^p(G)} \\
 \lesssim \sup_{g \in L^p(G) \atop \|\Delta^\gamma g\|_p \leq 1}  \| |\nabla f|\|_{L^{p'}(G)} \|\Delta^\gamma g\|_{L^p(G)} \leq \| |\nabla f|\|_{L^{p'}(G)}
 \end{multline*}
 by \eqref{Rpggbis}. The constant in the inequality above being independent of $f$, the lemma follows.

\ep

\subsection{Conditions on graphs}  \label{Sscondgraphs}

\begin{definition}
We say that a graph $(G,\mu)$ is {\bf doubling} - or satisfies \eqref{D} - if there exists $C>0$ such that for any $x\in G$ and any $r>0$,
\begin{equation} \tag{D} \label{D} V(x,2r) \leq C V(x,r).
\end{equation}
\end{definition}

There are few consequences of the doubling property that is worthy of comments. First, a doubling graph is uniformly locally finite, meaning that the numbers of neighbors of a vertex is (finite and) uniformly bounded. Second, doubling graphs have an upper bound in the dimension, in the sense that if $(G,\mu)$ is doubling, then there exists $C>0$ and $D_G>0$ such that
\begin{equation} \label{defDimension}
V(x,\lambda r) \leq C \lambda^{D_G} V(x,r) \qquad \text{ for } x\in G, \, \lambda \geq 1, \, r>0.
\end{equation}

A stronger variant of the doubling property is the Ahlfors regularity. We say that a graph is $D$-Ahlfors regular - or satisfies \ref{AR} - if there exists $C>0$ such that 
\begin{equation} \tag*{\normalfont (AR)$_D$} \label{AR}
C^{-1} (1+r)^D \leq V(x,r) \leq C(1+r)^D \qquad \text{ for } x\in G, \, r > 0.
\end{equation}

\begin{definition}
We say that $(G,\mu)$ satisfies \eqref{LB} if $(x,x)\in E_G$ for all $x\in G$ and there exists $\epsilon$ such that 
\begin{equation} \tag{LB} \label{LB}
\mu_{xy} \geq \epsilon m(x) \qquad \text{ for } (x,y)\in E_G.
\end{equation}
\end{definition}

The condition \eqref{LB} guarantees that the transition probability between two neighbor points cannot be too small, with a uniform constant. As for \eqref{D}, it implies that the graph is uniformly locally finite. The fact that $\mu_{xx}\geq \epsilon m(x)$, i.e. the fact that the random walk can stay on the same vertex, implies that $-1$ is not in the ($L^2$) spectrum of $P$, which in turn implies 
\begin{equation} \label{I+Pinvert}
\text{ The operator $(I+P)$ is invertible in $L^p(G)$ for all $p\in (1,\infty)$;}
\end{equation}
and the discrete semigroup $\{P^k\}_{k\in \N}$ is analytic, i.e. for any $\gamma \geq 0$ and any $p\in (1,\infty)$, there exists $C_{p,\gamma}$ such that
 \begin{equation} \label{Analyticity}
\|\Delta^\gamma P^{k-1} f \|_{L^p(G)} \leq \frac{C_{p,\gamma}}{k^\gamma} \|f\|_{L^p(G)} \qquad \text{ for } f\in L^p(G) \text{ and } k \in \N^*.
\end{equation}
Details for those implications, as well as characterizations of analyticity on graphs, can be found in \cite{FenLB,CSC90} and the reference therein.

\begin{definition}
Let $\beta \geq 2$. We say that the graph $(G,\mu)$ satisfies \ref{UE} if there exists $C,c>0$ such that
\begin{equation} \tag*{\normalfont (UE)$_{\beta}$} \label{UE}
p_{k-1}(x,y) \leq \frac{C}{V(x,k^{1/\beta})} \exp\left(- c \Big[\frac{\delta(x,y)^\beta}{k}\Big]^{\frac1{\beta-1}}  \right) \qquad \text{ for } (x,y)\in G^2, \, k\in \N^*.
\end{equation}
\end{definition}

Despite the complicated expression, the bound \ref{UE} and its complement 
\begin{equation} \tag*{\normalfont (LUE)$_{\beta}$} \label{LUE}
p_{k-1}(x,y) \geq \frac{c}{V(x,k^{1/\beta})} \exp\left(- C \Big[\frac{\delta(x,y)^\beta}{k}\Big]^{\frac1{\beta-1}}  \right) \qquad \text{ for } (x,y)\in G^2, \, k\in \N^*,
\end{equation}
 comes naturally for graphs with a self-similar fractal structure and a diffusion speed of order $k^{1/\beta}$, see for instance \cite{BP88,Kig89,BB92}, and \cite{GT01} for a characterization of \ref{UE} + \ref{LUE}. Notably, for any $D\geq 1$ and any $\beta \in [2,D+1]$, Barlow constructs in \cite{Barlow} an infinite graph that satisfies \ref{AR}, \ref{UE} and \ref{LUE}. 
 
 \smallskip
 
 The sub-Gaussian bounds \ref{UE} on the Markow kernel pass on the ``discrete time derivative'' of the kernel, as the following proposition shows.
 
 \begin{proposition}
 Let $(G,\mu)$ be a graph satisfying \eqref{D}, \eqref{LB}, and \ref{UE} for an $\beta \geq 2$. Then there exists $C,c>0$ such that
\begin{equation} \label{dkpk}
|\partial_k p_{k-1}(x,y)| \leq \frac{C}{k V(x,k^{1/\beta})} \exp\left(- c \Big[\frac{\delta(x,y)^\beta}{k}\Big]^{\frac1{\beta-1}}  \right) \qquad \text{ for } (x,y)\in G^2, \, k\in \N^*,
\end{equation}
where $\partial_k p_{k-1}(x,y) := p_{k}(x,y) - p_{k-1}(x,y) = \Delta[p_{k-1}(.,y)](x)$.
\end{proposition}

The proof of the above result can be found in \cite{Dungey09}, for instance.

\medskip

We finish the paragraph with a scaled Poincar\'e inequality on balls.

\begin{definition}
Let $q\in [1,\infty]$ and $s \geq 1$. We say that the graph $(G,\mu)$ satisfies the $L^q$ Poincaré inequality \ref{Pq} if there exists $C>0$ such that, for any $x\in G$, any $r>0$, and any function $u$ on $G$
\begin{equation} \tag*{\normalfont (P$_q$)$_{s}$} \label{Pq}
\left( \sum_{z\in B(x,r)} \big| u(z) - u_{B(x,r)} \big|^q m(z)  \right)^\frac1q \leq C r^s \left( \sum_{z\in B(x,r)} \big| \nabla u \big|^q(z) m(z)  \right)^\frac1q,
\end{equation}
where $u_{B(x,r)} := \frac1{V(x,r)} \sum_{z\in B(x,r)} u(z) m(z)$.
\end{definition}

\subsection{Vicsek graphs} \label{SsVicsek}

We need to recall the definition of tree graphs. Morally, they are the graphs with no ``loops''. They are introduced via the following definition.

\begin{definition}
A graph $(G,\mu)$ is a tree if the following holds: any closed path of length $N\geq 3$ has a repeating element. It means that for any $N\geq 3$ and any $(N+1)$-uple $(x_0,\dots,x_N)\in G^{N+1}$ satisfying $x_0=x_N$ and $(x_{i-1},x_i) \in E_G$ for all $i\in \{1,\dots,N\}$, there exists $1\leq i < j \leq N$ such that $x_i = x_j$. 
\end{definition}

When a graph is a tree, then, for any couple $(x,y)\in G^2$, there is a unique path joining $x$ to $y$ without any repeating element, and the length of this path is $\delta(x,y)$. Now, we are ready for our definition of Vicsek graphs.

\begin{definition}
For $D \geq 1$, we say that the graph $(G,\mu)$ is a $D$-dimensional Vicsek graph if $(G,\mu)$ is a tree that  satisfies \ref{AR}, \eqref{LB}, and {\normalfont (UE)}$_{D+1}$.
\end{definition}

We can see from \cite{Barlow} that Vicsek graphs are extreme type of graphs, as they have the slowest possible diffusion for a given growth rate. Let us point out that in the usual literature, Vicsek graphs have specific self-similar structures and also satisfy (LUE)$_{D+1}$. In our article, we decided to extend the definition a little bit\footnote{It is not clear to the author whether (LUE)$_{D+1}$ is a consequence of \ref{AR}, \eqref{LB}, and (UE)$_{D+1}$, or not.}, to highlight what we really need.  An example of Vicsek graph was given as Figure \ref{fig1} in the introduction.

\section{Tent spaces} \label{S3}

\subsection{Definitions and first properties.}

\label{Sstents}

In this section, $(G,\mu)$ is a graph satisfying \eqref{D}, and $\beta \geq 1$. In the rest of the paragraph, $\mathfrak f$, $\mathfrak g$ or $\mathfrak h$ denote (ordered) countable collections, and we write respectively $f_k$, $g_k$ or $h_k$ for the $k^{th}$ function of the collection.

\medskip

We recall the definition of the area and Carleson functionals.

\begin{definition}
The cone with vertex at $x\in G$ is
\[\Gamma(x) := \{(y,k) \in G\times \N^*| \, \delta^\beta(x,y) < k\}.\]
The Carleson region over $O \subset \Gamma$ is
\[\hat O := \{(y,k) \in G \times \N^*| \, \forall z\in O^c,  \delta^\beta(y,z) \geq k\}.\]
%$F\subset G$ is
%$$\check F = \bigcup_{x\in F} \gamma(x)$$
We then define the functionals $\mathcal A$ and $\mathcal C$ mapping collections of functions on $G$ to functions on $G$ by
\[\mathcal A( \mathfrak f)(x) := \left( \sum_{(y,k) \in G(x)} \frac{m(y)}{kV(x,k^{1/\beta})} |f_k(y)|^2 \right)^\frac{1}{2}\]
and
\[\mathcal C( \mathfrak f)(x)= \sup_{x\in B} \left(\frac{1}{V(B)} \sum_{(y,k)\in \hat B}\frac{m(y)}{k}|f_k(y)|^2 \right)^\frac{1}{2}.\]

For any $p\in [1,+\infty)$, the tent space $T^p(G)$ denotes the space of collections $\mathfrak f$ such that $\mathcal A(\mathfrak f) \in L^p(G)$. They are a Banach space when equipped with the norm $\|.\|_{T^p(G)}:= \|\mathcal A(.)\|_{L^p(G)}$. The tent space $T^\infty(G)$ is the space of collections $\mathfrak f$ such that $\mathcal C \mathfrak f \in L^\infty(G)$, and is equipped with the norm $\|.\|_{T^\infty(G)}:= \|\mathcal C(.)\|_{L^p(G)}$.
\end{definition}

The first result that we need is the fact that changing the aperture of the cones $\Gamma(x)$ will not change the tents spaces $T^{p}(G)$.

\begin{lemma} \label{lemonA}
For $i \in \N$, define 
\[\Gamma_{[i]}(x) := \{(y,k) \in G\times \N^*| \, \delta^\beta(x,y) < 2^ik\}\]
when $x\in G$, and then
\[\mathcal A_{[i]}( \mathfrak f)(x) := \left( \sum_{(y,k) \in \Gamma_{[i]}(x)} \frac{m(y)}{kV(x,k^{1/\beta})} |f_k(y)|^2 \right)^\frac{1}{2}\]
when $\mathfrak f$ is a collection of functions on $G$ and $x\in G$.

\smallskip

For all $p\in (0,\infty)$, there exists $C>0$ such that for any collection $\mathfrak f$, 
\[ \|\mathcal A_{[i]}( \mathfrak f)(x)\|_{L^p(G)} \leq C^i \|\mathcal A( \mathfrak f)(x)\|_{L^p(G)}.\]
\end{lemma}

\bp
The complicated part is to prove that 
\begin{equation} \label{equivA}
\|\mathcal A_{[2]}( \mathfrak f)(x)\|_{L^p(G)} \leq C_1 \|\mathcal A_{[0]}( \mathfrak f)(x)\|_{L^p(G)}.
\end{equation}
The argument relies on a good lambda argument which is now classical. The proof in the Euclidean case is found as Proposition 4 in \cite{CMS85}, but the proof only relies on the doubling property of $G$.

\medskip

Then, given $i\in \N$, we define $\tilde {\frak f} = \{\tilde f_k\}_{k\in \N^*}$ as $\tilde f_k := f_{\lfloor k/2^i\rfloor}$, with the convention that $f_0 \equiv 0$. Then we have, using \eqref{equivA},
\[ \|\mathcal A_{[i+2]}( \mathfrak f)(x)\|_{L^p(G)} \leq 2^{i/2} \|\mathcal A_{[2]}(\tilde{ \mathfrak f})(x)\|_{L^p(G)} \leq C_1 2^{i/2}  \|\mathcal A_{[0]}(\tilde{ \mathfrak f})(x)\|_{L^p(G)} \leq C_1 \|\mathcal A_{[i+1]}(\tilde{ \mathfrak f})(x)\|_{L^p(G)}.\]
The lemma follows.
\ep

\begin{theorem} \label{theoTent}
We have the following results on tent spaces:
\begin{itemize}
\item{\bf Carleson inequality.} There exists $C>0$ such that, for any collections $\mathfrak f$ and $\mathfrak g$ of functions on $G$ satisfying
  \[\mathcal A(\mathfrak f)(y) + \mathcal C(\mathfrak g)(y) < +\infty \qquad \text{ for any } y\in G,\]
we have
\[\sum_{(y,k)\in G \times \N^*} \frac{m(y)}{k} |f_k(y)g_k(y)| \lesssim \sum_{y\in G} \mathcal A(\mathfrak f)(y) \mathcal C(\mathfrak g)(y) \, m(y).\]

\smallskip

\item {\bf Duality.} If $p\in [1,\infty)$, the dual space of $T^p(G)$ for the bracket 
\[ \left< \mathfrak f, \mathfrak g \right> := \sum_{(y,k)\in G\times \N^*} f_k(y) g_k(y) \frac{m(y)}{k}\]
is $T^{p'}(G)$. In particular, 
\[ \|\mathcal A(\mathfrak f)\|_{L^p(G)}  \approx \sup_{\|\mathcal A(\frak g)\|_{p'} \leq 1} \sum_{(y,k)\in G\times \N^*} f_k(y) g_k(y) \frac{m(y)}{k}\]
when $p\in (1,\infty)$, and
\[ \|\mathcal A(\mathfrak f)\|_{L^1(G)}  \approx \sup_{ \|\mathcal C(\frak g)\|_{\infty} \leq 1} \sum_{(y,k)\in G\times \N^*} f_k(y) g_k(y) \frac{m(y)}{k},\]
In both cases, we can restrict the supremums to collections of compactly supported functions for which $f_k \equiv 0$ for a large enough $k$.

\smallskip

\item {\bf Equivalence between $\mathcal A$ and $\mathcal C$.} When $p\in (2,\infty)$, for any collection of functions $\mathfrak f$, we have
\[ \|\mathcal A(\mathfrak f)\|_{L^p(G)}  \approx  \|\mathcal C(\mathfrak f)\|_{L^p(G)},\]
where the constants are independent of $\mathfrak f$.
\end{itemize}
\end{theorem}

\bp
For tents spaces in the Euclidean settings, those results are classical, and can be found as Theorems 1, 2, and 3 in \cite{CMS85}. The proofs of \cite{CMS85} only rely of the doubling structure of the space, and can easily be adapted to our setting. In this particular setting, one can find the proof of the Carleson inequality and of $(T^1(G))^* = T^\infty(G)$ in Theorem D.3.8 of \cite{FeneuilThesis}.
\ep

\medskip

For $j\in \N$ and $q\in [1,\infty]$, let $\mathcal A^q_{(j)}$ and $\mathcal C^q_{(j)}$ be variants of $\mathcal A$ and $\mathcal C$ defined on collections of functions $\mathfrak f = \{f_k\}_{k\in \N^*}$ on $G$ as follows:
\[\mathcal A^q_{(j)} \mathfrak f(x):=  \left(\sum_{(y,k)\in \Gamma(x)}\frac{m(y)}{k V(x,k^{1/\beta})V(y,2^jk^{1/\beta})^\frac2q} \|f_k\|_{L^q(B(y,2^jk^{1/\beta}))}^2 \right)^\frac{1}{2}\]
and
\[\mathcal C^q_{(j)} \mathfrak f(x):= \sup_{r\in \N^*} \left(\frac{1}{V(x,r)} \sum_{(y,k)\in \widehat{B(x,r)}}\frac{m(y)}{k V(y,2^jk^{1/\beta})^{2/q}} \|f_k\|_{L^q(B(y,2^jk^{1/\beta}))}^2 \right)^\frac{1}{2}.\]
Moreover, we write $\mathcal A^p$ and $\C^p$ for $\mathcal A^q_{(0)}$ and $\mathcal C^q_{(0)}$ respectively.

\medskip

Theorem \ref{theoTent} can be adapted to those new functionals.

\begin{corollary} \label{theoTent2}
Let $q\in [1,\infty]$ and $j\in \N$. We have the following results 
\begin{itemize}
\item{\bf Carleson inequality.} There exists $C>0$ such that, for any collections $\mathfrak f$ and $\mathfrak g$ of functions on $G$ satisfying
  \[\mathcal A^q_{(j)}(\mathfrak f)(y) + \mathcal C^{q'}_{(j)}(\mathfrak g)(y) < +\infty \qquad \text{ for any } y\in G,\]
we have
\[\sum_{(y,k)\in G \times \N^*} \frac{m(y)}{k} |f_k(y)g_k(y)| \lesssim \sum_{y\in G} \mathcal A^q_{(j)}(\mathfrak f)(y) \mathcal C^{q'}_{(j)}(\mathfrak g)(y) \, m(y).\]

\smallskip

\item {\bf Duality.} When $p\in (1,\infty)$,
\[ \|\mathcal A^q_{(j)}(\mathfrak f)\|_{L^p(G)}  \approx \sup_{ \|\mathcal A^{q'}_{(j)}(\frak g)\|_{p'} \leq 1} \sum_{(y,k)\in G\times \N^*} f_k(y) g_k(y) \frac{m(y)}{k}\]
and when $p=1$
\[ \|\mathcal A^q_{(j)}(\mathfrak f)\|_{L^1(G)}  \approx \sup_{\|\mathcal C^{q'}_{(j)}(\frak g)\|_{\infty} \leq 1} \sum_{(y,k)\in G\times \N^*} f_k(y) g_k(y) \frac{m(y)}{k}.\]
In both cases, we can restrict the supremums to collections of compactly supported functions for which $f_k \equiv 0$ for a large enough $k$.

\smallskip

\item {\bf Equivalence between $\mathcal A$ and $\mathcal C$.} When $p\in (2,\infty)$, we have  for any collection of functions $\mathfrak f$,
\[ \|\mathcal A^q_{(j)}(\mathfrak f)\|_{L^p(G)}  \approx  \|\mathcal C^q_{(j)}(\mathfrak f)\|_{L^p(G)}.\]
\end{itemize}
Note that in all the theorem, due to our definition of $\mathcal A^q_{(j)}$ and $\mathcal C^q_{(j)}$, the constants can be chosen independent of $q$ and $j$. But this independence will not play a role in our future proofs.
\end{corollary}

\bp
The first and third point are consequences of Theorem \ref{theoTent} applied to the collections $\bar{\frak f}= \{\bar f_k\}_{k\in \N^*}$ and $\bar{\frak g}= \{\bar g_k\}_{k\in \N^*}$ defined as
\[ \bar f_k(y) := \frac{1}{V(y,2^jk^{1/\beta})^{1/q}} \|f_k\|_{L^q(B(y,2^jk^{1/\beta}))} \ \text{ and } \ \bar g_k(y) := \frac{1}{V(y,2^jk^{1/\beta})^{1/q'}} \|g_k\|_{L^{q'}(B(y,2^jk^{1/\beta}))}.\]

So only the duality is not immediate, but is still a consequence of the duality in Theorem \ref{theoTent}. Let us only do the case $p\in (1,\infty)$ and $j=0$, since the cases $p=1$ and $j\in \N^*$ are similar. If we keep the notation introduced above, 
\[ \|\A^q(\mathfrak f)\|_{L^p(G)} = \|\A(\bar{ \mathfrak f})\|_{L^p(G)} \lesssim \sum_{(y,k)\in G\times \N^*} \bar f_k(y) \bar g_k(y) \frac{m(y)}{k} \]
for a collection $\bar{\frak g} = \{\bar g_k\}_{k\in \N^*}$ satisfying $\|\A(\bar{\frak g})\|_{L^p(G)} \leq 1$. But with the duality on the spaces $L^q(B(y,k^{1/\beta}))$, we also have the existence of $h_{y,k}$ verifying $\|h_{y,k}\|_{L^{q'}(B(y,k^{1/\beta}))} \leq V(y,k^{1/\beta})^{-1/q}$ such that 
\[  \bar f_k(y) = \sum_{z\in B(y,k^{1/\beta})} f_k(z) h_{y,k}(z) m(z).\]
We define $\mathfrak g =\{g_k\}_{k\in \N^*}$ as
\[ g_k(z):= \sum_{y\in B(z,k^{1/\beta})} \bar g_k (y) h_{y,k}(z) m(y).\]
We let the reader check that 
\[  \sum_{(y,k)\in G\times \N^*} \bar f_k(y) \bar g_k(y) \frac{m(y)}{k} =  \sum_{(z,k)\in G\times \N^*} f_k(z) g_k(z) \frac{m(z)}{k}, \]
meaning that 
\[ \|\A^q(\mathfrak f)\|_{L^p(G)} \lesssim \sum_{(z,k)\in G\times \N^*} f_k(z) g_k(z) \frac{m(z)}{k}.\]

So the theorem will be proved once we establish that 
\begin{equation} \label{claimduality}
\|\A^{q'}(\frak g)\|_{L^{p'}(G)} \lesssim 1. 
\end{equation}
To prove the claim, first observe that for any function $\varphi$ on $G$,
\begin{multline*}
\sum_{B(a,k^{1/\beta})} g_k(z) \varphi(z) m(z) = \sum_{z\in B(a,k^{1/\beta})} \varphi(z) m(z) \sum_{y\in B(z,k^{1/\beta})} \bar g_k(y) h_{y,k}(z) m(y) \\
= \sum_{y\in B(a,2k^{1/\beta})}  \bar g_k(y) m(y) \sum_{z\in B(a,k^{1/\beta}) \cap B(y,k^{1/\beta})} \varphi(z) h_{y,k}(z) m(z)
\end{multline*}
So, since $\|h_{y,k}\|_{L^{q'}(B(y,k^{1/\beta}))} \leq V(y,k^{1/\beta})^{-1/q} \lesssim V(a,k^{1/\beta})^{-1/q}$ when $y\in B(a,2k^{1/\beta})$, we deduce that
\begin{multline*}
 \left| \sum_{B(a,k^{1/\beta})} g_k(z) \varphi(z) m(z) \right| \leq \|\varphi\|_{L^q(B(a,k^{1/\beta}))} \sup_{y\in B(a,2k^{1/\beta})}  \|h_{y,k}\|_{L^{q'}(B(a,k^{1/\beta}))} \sum_{y\in B(a,2k^{1/\beta})}  |\bar g_k(y)| m(y) \\
 \leq \|\varphi\|_{L^q(B(a,k^{1/\beta}))} V(a, k^{1/\beta})^{1/q'} \sum_{y\in B(a,2k^{1/\beta})}  |\bar g_k(y)| m(y),
 \end{multline*}
 for all function $\varphi$ on $G$. By duality, we deduce that
\begin{multline*} 
\frac{\|g_k\|_{L^{q'}(B(a,k^{1/\beta}))}}{V(a, k^{1/\beta})^{1/q'}} \lesssim \frac{1}{V(a,k^{1/\beta})} \sum_{y\in B(a,2k^{1/\beta})}  |\bar g_k(y)| m(y) \\ \lesssim \left( \frac{1}{V(a,2k^{1/\beta})} \sum_{y\in B(a,2k^{1/\beta})}  |\bar g_k(y)|^2 m(y) \right)^\frac12.
 \end{multline*}
Using Fubini's lemma gives then that 
\[ \|\A^{q'}(\frak g)\|_{L^{p'}(G)} \lesssim  \|\A_{[2]}(\bar{\frak g})\|_{L^{p'}(G)} \lesssim 1\]
by Lemma \ref{lemonA}.
\ep

\subsection{Equivalence of area functionals}

\begin{lemma}  \label{lemLP2A}
Let $G$ be a doubling graph, $q\in [1,\infty]$, and $p\in (0,\infty)$. Then there exists $C>0$ and $D>0$ such that, for any $j\in \N^*$, any collection of functions $\mathfrak f = \{f_k\}_{k\in \N^*}$, and any $x\in G$, we have
\[   C^{-1} 2^{-jD} \big\|\A^q (\mathfrak f)\big\|_{L^p(G)} \leq  \big\|\A^q_{(j)}(\mathfrak f)\big\|_{L^p(G)} \leq C2^{jD} \big\|\A^q(\mathfrak f)\big\|_{L^p(G)}.\]
\end{lemma}

\bp First, note that we have with \eqref{defDimension} that $V(y,2^jk^{1/\beta}) \lesssim 2^{jD_G} V(y,k^{1/\beta})$, and hence
\[\A^q\mathfrak f(x) \leq C2^{jD_G/q} \A^q_{(j)}\mathfrak f(x),\]
with constants independent on $x\in G$, and the collection $\mathfrak f$. So the difficult bound is the second one. 

\medskip

Take $x\in G$ and $k\in \N^*$. Assume that $q<+\infty$, since the case $q=\infty$ is an easier variant. We easily have with the doubling property that
\begin{multline*}
I_k:= \sum_{y\in B(x,k^{1/\beta})} \frac{m(y)}{k V(x,k^{1/\beta})V(y,2^jk^{1/\beta})^\frac2q} \|f_k\|_{L^q(B(y,2^jk^{1/\beta}))}^2 \\
\lesssim  \frac1k  \left|  \frac{1}{V(x,2^{j+1}k^{1/\beta})}  \sum_{z\in B(x,2^{j+1}k^{1/\beta})} |f_k(z)|^q m(z) \right|^{\frac2q}
\end{multline*}
Take a ball $\bar B^k_x \subset B(x,2^{j+1}k^{1/\beta})$ of radius $k^{1/\beta}/2$ such that 
\[ \sum_{z\in \bar B_{x}^k}|f_k(z)|^q m(z) \gtrsim \frac{1}{2^{jD_G}}  \sum_{z\in B(x,2^{j+1}k^{1/\beta})}|f_k(z)|^q m(z),\]
which is possible thanks to the Vitali lemma, and so we have, using the doubling property and \eqref{defDimension} when necessary,
\begin{multline*}
I_k \lesssim \frac{2^{2jD_G/q}}{k}  \left|  \frac{1}{V(x,2^{j+1}k^{1/\beta})} \sum_{z\in \bar B^k_x}|f_k(z)|^q m(z) \right|^{\frac2q} \\
\lesssim  \frac{2^{2jD_G/q}}{k} \frac{1}{V(\bar B^k_x)}  \sum_{y\in \bar B^k_x} m(y) \left| \frac1{V(x,2^{j+1}k^{1/\beta})}  \sum_{z\in B(y,k^{1/\beta})}|f_k(z)|^q m(z) \right|^{\frac2q}  \\
 %\lesssim \frac{2^{4jD_G/q}}{k}  \left| \frac{1}{V(\bar B^k_x)} \sum_{z\in \bar B^k_x}|f_k(z)|^q m(z) \right|^{\frac2q} \\
\lesssim  2^{jD_G(1+4/q)} \sum_{y\in B(x,2^{j+2}k^{1/\beta})} \frac{m(y)}{kV(x,k^{1/\beta})} \left|  \frac1{V(y,k^{1/\beta})}  \sum_{z\in B(y,k^{1/\beta})}|f_k(z)|^q m(z) \right|^{\frac2q} \\
\lesssim 2^{jD_G(1+4/q)} \sum_{y\in B(x,2^{j+2}k^{1/\beta})} \frac{m(y)}{kV(x,k^{1/\beta})} \big|\bar f_k(y)\big|^2
\end{multline*}
where
\[ \bar f_k(y) := \frac{1}{V(y,k^{1/\beta})^{1/q}} \|f_k\|_{L^q(B(y,k^{1/\beta}))}.\]

\medskip

Now, we sum over $k\in \N^*$ to get that 
\[ \A_{(j)}^q\big(\frak f\big)(x) \lesssim 2^{jD_G(1/2+2/q)}  \A_{[j]}  \big(\bar{\mathfrak f}\big)(x).\]
By Lemma \ref{lemonA}, we can find $D_p$ such that 
\begin{multline*} 
\big\|\A_{(j)}^q\big(\frak f\big)\big\|_{L^p(G)} \lesssim 2^{jD_G(1/2+2/q)} \big\|\A_{[j]} (\bar{\mathfrak f})\big\|_{L^p(G)} \\
\lesssim 2^{jD_G(1/2+2/q) + jD_p} \big\|\A(\bar{\mathfrak f})\big\|_{L^p(G)} = 2^{jD} \big\|\A^q(\mathfrak f)\big\|_{L^p(G)}
\end{multline*}
for a $D$ independent of $\frak f$. The lemma follows.
\ep

\begin{lemma}  \label{lemLP2}
Let $G$ be a doubling graph, $q\in [1,\infty]$ and $p\in (0,\infty)$. Then there exists $C>0$ and $D>0$ such that, for any $j\in \N^*$, any collection of functions $\mathfrak f = \{f_k\}_{k\in \N^*}$, and any $x\in G$, we have
\[ C^{-1}2^{-jD} \C^p(\mathfrak f)(x) \leq \C^q_{(j)}( \mathfrak f)(x) \leq C2^{jD} \C^q(\mathfrak f)(x).\]
\end{lemma}

\bp
We do not need to be particularly subtle for the proof of the lemma. Moreover, we write the proof for the case $q<+\infty$, as the case $q=\infty$ is a simpler variant.

\medskip

We write $D_G$ for an upper bound on the ``dimension'' given in \eqref{defDimension}. The first inequality of the lemma is clear, as the doubling property, or more exactly \eqref{defDimension}, gives that $V(y,2^jk^{1/\beta}) \lesssim 2^{jD_G} V(y,k^{1/\beta})$ and hence
\[\C^p\mathfrak f(x) \leq C2^{jD_G/q} \C^p_{(j)}\mathfrak f(x).\]

\medskip

Let us turn to the second inequality, which is a bit less clear. Let $\mathfrak f = \{f_k\}_{k\in \N^*}$ be a collections of functions on $G$. Pick a ball $B=B(x,r) \subset G$ and $k \in \N^*$ such that $k\leq r^\beta$. By the Vitali lemma, we can find a covering $\{B_i^k\}_{1\leq i \leq N_k}$ of $B$ by $N_k\leq Cr/k^{1/\beta}$ balls of radius $k^{1/\beta}$. We can also ensure that $\sum_{1\leq i \leq N} \1_{2^{j+2}B_i^k}$ is bounded by $C2^{jD_G}$, where $C$ is independent of $B$ and $k$.  It ensures that
\begin{multline*}
I_k := \sum_{y\in B} \left|  \frac1{V(y,2^jk^{1/\beta})} \sum_{z\in B(y,2^jk^{1/\beta})} |f_k(z)|^q m(z) \right|^{\frac2 q} \frac{m(y)}{k} \\
\leq \sum_{i=1}^{N_k} \sum_{y\in B^k_i} \left|  \frac1{V(y,2^jk^{1/\beta})} \sum_{z\in B(y,2^jk^{1/\beta})}|f_k(z)|^q m(z) \right|^{\frac2 q} \frac{m(y)}{k}
\end{multline*}
But since $B(y,2^jk^{1/\beta}) \subset 2^{j+1}B^k_i$ whenever $y\in B^k_i$, we have
\[I_k \lesssim \sum_{i=1}^{N_k} V(B^k_i) \left|  \frac1{V(2^{j+1}B^k_i)} \sum_{z\in 2^{j+1}B_{i}^k} |f_k(z)|^q m(z)\right|^{\frac2 q} \frac{m(y)}{k}. \]
For each $i\in \{1,\dots,N_k\}$, take $\bar B_i^k \subset 2^{j+2}B_i^k$ a ball of radius $k^{1/\beta}/2$ such that 
\[ \sum_{z\in \bar B_{i}^k}|f_k(z)|^q m(z) \gtrsim \frac{1}{2^{jD_G}}  \sum_{z\in 2^{j+1}B_{i}^k}|f_k(z)|^q m(z),\]
and since $V(2^{j+1}B_i^k) \lesssim 2^{jD_G} V(\bar B_i^k)$, we get that
\begin{multline*}
I_k \lesssim 2^{jD_G(4/q+1)} \sum_{i=1}^{N_k} V(\bar B^k_i) \left|  \frac1{V(\bar B^k_i)} \sum_{z\in  \bar B_{i}^k} |f_k(z)|^q m(z) \right|^{\frac2 q} \frac{m(y)}{k} \\
\lesssim 2^{jD_G(1+4/q)} \sum_{i=1}^{N_k} \sum_{y\in \bar B_i^k} \left|  \frac1{V(y,k^{1/\beta})} \sum_{z\in  B(y,k^{1/\beta})} |f_k(z)|^q m(z) \right|^{\frac2 q} \frac{m(y)}{k} \\
\leq 2^{jD_G(1+4/q)} \sum_{y\in 2^{j+2}B}  \sum_{i=1}^{N_k} \1_{2^{j+2}B^k_i}(y) \left|  \frac1{V(y,k^{1/\beta})} \sum_{z\in  B(y,k^{1/\beta})} |f_k(z)|^q m(z) \right|^{\frac2 q} \frac{m(y)}{k} \\
\lesssim 2^{jD_G(2+4/q)} \sum_{y\in 2^{j+2}B} \left|  \frac1{V(y,k^{1/\beta})} \sum_{z\in  B(y,k^{1/\beta})} |f_k(z)|^q m(z) \right|^{\frac2 q} \frac{m(y)}{k},
\end{multline*}
where we use the fact that $\bar B_i^k \subset 2^{j+2}B_i^k$ for the third line, and the fact that $\sum_{1\leq i \leq N} \1_{2^{j+2}B_i^k} \lesssim 2^{jD_G}$ for the last one.

\medskip

Now, we sum over the $k\in \N^*$ and we get
\begin{multline*}
 \left(\frac{1}{V(x,r)} \sum_{k=1}^{r^\beta} \sum_{y\in B(x,r)} \frac{m(y)}{k}\left|  \frac1{V(y,2^jk^{1/\beta})} \sum_{z\in B(y,2^jk^{1/\beta})} |f_k(z)|^q m(z) \right|^{\frac2 q} \right)^\frac{1}{2} \\
 \lesssim 2^{jD_G(1+2/q)}  \left(\frac{1}{V(x,r)} \sum_{k=1}^{r^\beta}  \sum_{y\in 2^{j+2}B} \left|  \frac1{V(y,k^{1/\beta})} \sum_{z\in  B(y,k^{1/\beta})} |f_k(z)|^q m(z) \right|^{\frac2 q} \frac{m(y)}{k} \right)^\frac{1}{2}  \\
 \lesssim 2^{jD_G(3/2+2/q)}  \left(\frac{1}{V(x,2^{j+2}r)} \sum_{(y,k) \in \widehat{2^{j+2}B}} \left|  \frac1{V(y,k^{1/\beta})} \sum_{z\in  B(y,k^{1/\beta})} |f_k(z)|^q m(z) \right|^{\frac2 q} \frac{m(y)}{k} \right)^\frac{1}{2} \\
 \leq 2^{jD_G(3/2+2/q)} \mathcal C^q \mathfrak f(x),
\end{multline*}
Using again the doubling property to control the ratio $V(x,2^{j+2}r)/V(x,r)$. Taking the supremum on $r\in \N^*$ gives 
\[\mathcal C^q_{(j)} \mathfrak f(x) \lesssim 2^{jD_G(3/2+2/q)} \mathcal C^q \mathfrak f(x)\]
as desired for the lemma.
\ep

\subsection{Horizontal Lusin functional}

\begin{definition} \label{defLa}
 Define, for $\alpha>0$, the quadratic functionals $g_\alpha$ and $L_\alpha$ on $L^2(G)$ by
 $$g_\alpha f(x) :=  \left( \sum_{k\geq 1} k^{2\alpha-1} |\Delta^{\alpha} P^{k-1} f(y)|^2 m(x) \right)^{\frac{1}{2}}$$
 and
$$L_\alpha f(x) :=  \left( \sum_{(y,k) \in \Gamma(x)} \frac{k^{2\alpha-1}}{V(x,k)} |\Delta^{\alpha} P^{k-1} f(y)|^2m(y) \right)^{\frac{1}{2}} = \mathcal A(\{k\Delta^\alpha P^{k-1} f\}_{k\in \N^*})(x),$$
where $\Gamma(x) = \left\{ (y,k) \in G \times \N^*, \, \delta(x,y)^\beta < k \right\}$, as before.
\end{definition}

\begin{theorem} \label{theoga}
Let $(G,\mu)$ be a graph as defined in Subsection \ref{defgraphs}. If $G$ satisfies \eqref{LB} and $\alpha>0$, then the Littlewood-Paley functional $g_\alpha$ is bounded on $L^p$ for all $p\in (1,\infty)$. By duality, we also have the reverse bound. It means that there exists $C_{\alpha,p}$ such that 
\[ C_{\alpha,p}^{-1} \|f\|_{L^p(G)} \leq  \|g_\alpha(f)\|_{L^p(G)} \leq C_{\alpha,p} \|f\|_{L^p(G)}  \]
for any $f\in L^p(G)$.
\end{theorem}

\bp
For instance, the proof can be found in \cite[Chapter IV]{SteinLP} or \cite[Theorem 3.3]{LMX12} for the case $\alpha = 1$, and \cite[Theorem 3.3]{ALM} for the extension to other values of $\alpha>0$.
\ep

\begin{theorem} \label{theoLa}
Let $(G,\mu)$ be a graph as defined in Subsection \ref{defgraphs}. 
If $\alpha >0$ and  $G$ satisfies \eqref{LB}, \eqref{D} and \ref{UE} for some $\beta \geq 2$, then the Lusin functional $L_\alpha$ is bounded on $L^p$ for all $p\in (1,\infty)$. By duality, we also have the reverse bound. It means that there exists $C_{\alpha,p}$ such that 
\[ C_{\alpha,p}^{-1} \|f\|_{L^p(G)} \leq  \|L_\alpha(f)\|_{L^p(G)} \leq C_{\alpha,p} \|f\|_{L^p(G)}  \]
for any $f\in L^p(G)$.
\end{theorem}

\bp
Assume that $x\in G$, $y\in G$, and $k\in \N^*$ are such that $\delta(x,y)^\beta < 9k$. We claim that for any $g\in L^2(G)$, we have
\begin{equation} \label{claimPkg}
 |P^{k-1} g(x)| \leq \mathcal M(g)(y),\end{equation}
where $\mathcal M$ is the uncentered maximal function 
\[ \mathcal M(g)(x) := \sup_{B\ni x} \sum_{z\in B} |g(z)| \, m(z).\] 
Indeed, 
\begin{multline*}
|P^{k-1}g(x)| = \Big| \sum_{z\in G} p_{k-1}(x,z) g(z) m(z) \Big| \\ \lesssim \frac1{V(x,k^{1/\beta})} \sum_{z\in G} |g(z)| \exp\left( - c \Big[ \frac{\delta(x,y)^\beta}{k} \Big]^{\frac1{\beta-1}} \right) m(z)
\end{multline*}
We partition the graph $G$ into annulus. If $C_1(x,k^{1/\beta}) = B(x,4k^{1/\beta})$ and $C_j(x,k^{1/\beta}) =  B(x,2^{j+1}k^{1/\beta}) \setminus B(x,2^{j}k^{1/\beta})$, we obtain then that 
\begin{multline*}
|P^{k-1}g(x)|  \lesssim  \sum_{j\geq 1}  \frac1{V(x,k^{1/\beta})} \sum_{C_j(x,k^{1/\beta})} |g(z)| \exp\left( - c \Big[ \frac{\delta(x,y)^\beta}{k} \Big]^{\frac1{\beta-1}} \right) m(z) \\
\leq  \sum_{j\geq 1}   \frac{V(x,2^{j+1}k^{1/\beta})}{V(x,k^{1/\beta})} \exp\left( - c 2^{\frac{j\beta}{\beta-1}} \right)   \frac1{V(x,2^{j+3}k^{1/\beta})} \sum_{B(x,2^{j+3}k^{1/\beta})} |g(z)| m(z) \\
\lesssim \mathcal M(g)(y) \sum_{j\geq 1}   \frac{V(x,2^{j+3}k^{1/\beta})}{V(x,k^{1/\beta})} \exp\left( - c 2^{\frac{j\beta}{\beta-1}} \right)  \lesssim \mathcal M(g)(y),
\end{multline*}
since $y\in B(x,2^{j+3}k^{1/\beta})$ for all $j\geq 1$, and 
\[ \sum_{j\geq 1} \frac{V(x,2^{j+1}k^{1/\beta})}{V(x,k^{1/\beta})} \exp\left( - c 2^{\frac{j\beta}{\beta-1}} \right)  \lesssim    \sum_{j\geq 1} 2^{jD_G} \exp\left( - c 2^{\frac{j\beta}{\beta-1}} \right)  \lesssim 1\]
by \eqref{defDimension}.  The claim \eqref{claimPkg} follows.

\medskip

From there, we have then that, if $l=(k-1)/2$ or $k/2-1$ (depending whether $k$ is odd or even),
\begin{multline*}
\frac1{V(y,k^{1/\beta})} \sum_{z\in B(y,k^{1/\beta})} |\Delta^\alpha P^{k-1} f(z) |^2 m(z) = \frac1{V(y,k^{1/\beta})} \sum_{z\in B(y,k^{1/\beta})} |P^{l}[\Delta^\alpha P^{k-1-l} f(z)]|^2 m(z) \\
\lesssim  \frac1{V(y,k^{1/\beta})} \sum_{z\in B(y,k^{1/\beta})}  |\mathcal M(\Delta^\alpha P^{k-1-l} f)(x)|^2 m(z) =  |\mathcal M(\Delta^\alpha P^{k-1-l} f)(x)|^2 \end{multline*}
Summing over $k\in \N^*$ yields
\begin{multline*}
L_\alpha f(y) := \left( \sum_{k\geq 1} k^{2\alpha-1} \frac1{V(y,k^{1/\beta})} \sum_{z\in B(y,k^{1/\beta})} |\Delta^\alpha P^{k-1} f(z) |^2 m(z) \right)^\frac12 \\ \lesssim \left( \sum_{k\geq 1} k^{2\alpha-1}  |\mathcal M(\Delta^\alpha P^{\lfloor (k-1)/2\rfloor } f)(y)|^2 \right)^\frac12 \approx  \left( \sum_{k\geq 1} k^{2\alpha-1}  |\mathcal M(\Delta^\alpha P^{k-1} f)(y)|^2 \right)^\frac12 , 
\end{multline*}
where the last equivalence is obtained by reindexing the sum. After taking the $L^p$ norm, the Fefferman-Stein  inequality (whose proof in $\R^d$ can be found in \cite[Chapter 2, Theorem 1]{Stein93} and which extends easily to any doubling spaces as it relies on the Calder\'on-Zygmund decompostion) gives, for $p\in (1,\infty)$, that
\begin{multline*}
\|L_\alpha f\|_{L^p} \lesssim \left( \sum_{y\in G} \left( \sum_{k\geq 1} k^{2\alpha-1}  |\mathcal M(\Delta^\alpha P^{k-1-l} f)(y)|^2 m(y) \right)^\frac p2 \right)^\frac1p \\
\lesssim  \left( \sum_{y\in G} \left( \sum_{k\geq 1} k^{2\alpha-1}  |\Delta^\alpha P^{k-1-l} f(y)|^2 m(y) \right)^\frac p2 \right)^\frac1p = \|g_\alpha f\|_{L^p}.
\end{multline*}
Thanks to Theorem \ref{theoga}, we deduce the $L^p$-boundeness of the Lusin functional $L_\alpha$ for all $p\in (1,\infty)$. 

\medskip

The lower bound can be deduced by duality. Let us define the variant of the Lusin functional
\[\tilde L_\alpha f(x):=  \left( \sum_{(y,k) \in \Gamma(x)} \frac{c_k^{(2\alpha)}}{V(y,k)} |(I-P^2)^{\alpha} P^{k-1} f(y)|^2m(y) \right)^{\frac{1}{2}}\]
where $c_k^{(2\alpha)}$ are the coefficients of the power series 
\[ (1-z)^{-2\alpha} = \sum_{k\geq 1} c_k^{(2\alpha)} z^{k-1}.\]
By the doubling property \eqref{D}, we have $V(x,k^{1/\beta}) \approx V(y,k^{1/\beta})$ whenever $y\in B(x,k^{1/\beta})$. Moreover, Lemma B.1 in \cite{Fen15b} shows that $c_k^{(2\alpha)} \approx k^{2\alpha-1}$. All this proves that  
\[ \tilde L_\alpha f(x) \approx L_\alpha (I+P)^\alpha f(x)\]
for all function $f$ on $G$ and all $x\in G$. However,
Observe that $\tilde L_\alpha f$ is a better substitute of $L_\alpha$ because it is an isometry on $L^2$. Indeed,  Fubini's theorem yields that  $\|\tilde L_\alpha f\|_{2}^2 =  \|\tilde g_\alpha f\|_{2}^2$, where
\[\tilde g_\alpha f(x):=  \left( \sum_{k\geq 1}  c_k^{(2\alpha)} |(I-P^2)^{\alpha} P^{k-1} f(y)|^2m(x) \right)^{\frac{1}{2}}\]
and Lemma 3.1 in \cite{Fen15b} shows that $\|\tilde g_\alpha f\|_{2}^2 = \|f\|_{2}^2$, meaning that $\tilde L_\alpha$ is an isometry, and thus, for $f,h\in L^2(G)$
\[ \sum_{x\in G} f(x)h(x) m(x) = \sum_{k\geq 1} c_k^{(2\alpha)} \sum_{x\in G} \sum_{y\in B(x,k^{1/\beta})} \frac1{V(y,k^{1/\beta})} [\Delta^\alpha P^{k-1} f(y)] [\Delta^\alpha P^{k-1} h(y)]  m(y)\]
So if $f\in L^2(G) \cap L^p(G)$ with $p\in (1,\infty)$,
\begin{multline*}
\|f\|_{L^p}  = \sup_{h\in L^2\cap L^{p'}  \atop \|h\|_{p'} \leq 1} \sum_{x\in G} f(x)h(x) m(x) \\
=   \sup_{h\in L^2\cap L^{p'} \atop \|h\|_{p'} \leq 1} \sum_{k\geq 1} c_k^{(2\alpha)} \sum_{x\in G} \sum_{y\in B(x,k^{1/\beta})} \frac1{V(y,k^{1/\beta})} [\Delta^\alpha P^{k-1} f(y)] [\Delta^\alpha P^{k-1} h(y)] \\
\leq \|\tilde L_\alpha f\|_{L^p} \sup_{h\in L^2\cap L^{p'} \atop \|h\|_{p'} \leq 1}  \|\tilde L_\alpha h\|_{L^{p'}} \lesssim \|L_\alpha (I+P)^\alpha f\|_{L^p} \sup_{h\in L^2\cap L^{p'} \atop \|h\|_{p'} \leq 1}  \|L_\alpha (I+P)^\alpha h\|_{L^{p'}} \\  \lesssim \|L_\alpha (I+P)^\alpha f\|_{L^p}
\end{multline*}
by the $L^{p'}$-boundedness of $L_\alpha$ and $(I+P)^\alpha$. Finally, since $G$ satisfies \eqref{LB}, the operator $(I+P)$ is invertible in $L^p(G)$, as discussed previously in \eqref{I+Pinvert}. It means that, for any $f\in L^2(G) \cap L^p(G)$,
\[ \|f\|_{L^p} \lesssim \|(I+P)^{-\alpha} f\|_{L^p} \lesssim \|L_\alpha f\|_{L^p}.\]
The theorem follows by density of $L^2(G) \cap L^p(G)$ in $L^p(G)$.
\ep

\section{Consequences of Poincar\'e inequalities} \label{S4}

\subsection{Estimates on $P^{k-1}\nabla^*$}

\begin{lemma} \label{lemLPq}
Let $(G,\mu)$ be a graph satisfying \eqref{LB}, \eqref{D}, \ref{UE} for a $\beta \geq 2$, and \ref{Pq} for some $q\in [1,\infty]$ and $s \geq 1$. Then there exists $C,c>0$ such that, for any function $f \in L^q_{loc}(G)$, any $k\in \N^*$, and any $y\in G$, we have that
\[ [P^{k-1} \nabla^* \nabla f](y) \leq Ck^{\frac{s}\beta - 1} \sum_{j\in \N^*} e^{-c\lambda^j}\left( \frac1{V(y,2^jk^{1/\beta})} \sum_{x\in B(y,2^jk^{1/\beta})} |\nabla f|^q(x) m(x) \right)^\frac1q,\]
 where $\lambda = 2^{\beta/(\beta-1)}$.
\end{lemma}

\bp
We use the fact that  $\nabla^*\nabla=\Delta$, and we have
\begin{multline*} 
[P^{k-1} \nabla^*\nabla f](y) = [\Delta P^{k-1} f](y) \\ = \sum_{x\in G} \Delta[p_{k-1}(.,y)](x) \, f(x) \, m(x) = \sum_{x\in G} [\partial_k p_{k-1}(.,y)](x) \, f(x) \, m(x) \\
= \sum_{x\in G} [\partial_k p_{k-1}(.,y)](x) \, [f(x)-f_{B(y,k^{1/\beta})}] \, m(x)
\end{multline*}
where $f_{B(y,k^{1/\beta})} := \frac1{V(y,k^{1\beta})} \sum_{B(y,k^{1/\beta})} f(z) m(z)$, and where the last equality holds because $\sum_{x\in G} [\partial_k p_{k-1}(.,y)](x) m(x) = 0$ -  a simple consequence of \eqref{sumpk=1}. Let now, as usually, $C_1(y,k) := B(y,2k^{1/\beta})$ and $C_j(y,k):= B(y,2^{j}k^{1/\beta}) \setminus B(y,2^{j-1}k^{1/\beta})$ for $j\geq 2$, so that we can write
\[\left| [P^{k-1} \nabla^*\nabla f](y)\right|  \leq \sum_{j\in \N^*} \sum_{x\in C_j(y,k)} \big|[\partial_k p_{k-1}(.,y)](x)\big|\, \big|f(x)-f_{B(y,k^{1/\beta})}\big| \, m(x) \]
But thanks to the pointwise bounds \eqref{dkpk} on the discrete time derivative, we have for $x\in C_j(y,k)$,
\[|\partial_kp_{k-1}(x,y)| \lesssim \frac{1}{kV(y,k^{1/\beta})} \exp(-c \lambda^{j-1}) \lesssim \frac{1}{2^{j(s+D_G)}kV(y,2^{j}k^{1/\beta})} \exp(-c' \lambda^j),\]
for $0<c'<c/\lambda$, because the doubling property \eqref{defDimension} of $G$ gives
\[ \frac{2^{j(s+D_G)}kV(y,2^{j}k^{1/\beta})}{kV(y,k^{1/\beta})} \lesssim 2^{j(s + 2D_G)} \lesssim \exp\big([c/\lambda-c']\lambda^j\big).\]
Consequently
\begin{equation} \label{PkN*bd1}
\left| [P^{k-1} \nabla^*\nabla f](y) \right|  \lesssim \frac1{k} \sum_{j\in \N^*}  \exp(-c' \lambda^j) 2^{-j(\gamma+D_G)}  \Big( \underbrace{\displaystyle\frac1{V(y,2^{j}k^{1/\beta})}\sum_{x\in C_j(y,k)} |f(x)-f_{B(y,k^{1/\beta})}| m(x) }_{T_j}\Big).
\end{equation}
We bound the terms $T_j$ with the help of the Poincar\'e inequality. Indeed,
\begin{multline*}
T_j \leq \frac1{V(y,2^{j}k^{1/\beta})}\sum_{x\in B(y,2^jk^{1/\beta})} |f(x)-f_{B(y,2^jk^{1/\beta})}| m(x)  + |f_{B(y,k^{1/\beta})} - f_{B(y,2^jk^{1/\beta})} | \\
\leq \Big( \frac1{V(y,2^{j}k^{1/\beta})} + \frac1{V(y,k^{1/\beta})} \Big) \sum_{x\in B(y,2^jk^{1/\beta})} |f(x)-f_{B(y,2^jk^{1/\beta})}| m(x) \\
\lesssim  \frac{1+2^{jD_G}}{V(y,2^{j}k^{1/\beta})} \sum_{x\in B(y,2^jk^{1/\beta})} |f(x)-f_{B(y,2^jk^{1/\beta})}| m(x) \\
\lesssim 2^{jD_G} \left( \frac{1}{V(y,2^{j}k^{1/\beta})} \sum_{x\in B(y,2^jk^{1/\beta})} |f(x)-f_{B(y,2^jk^{1/\beta})}|^q m(x) \right)^{\frac1q} \\
\lesssim 2^{j(D_G+s)} k^{s/\beta} \left( \frac{1}{V(y,2^{j}k^{1/\beta})} \sum_{x\in B(y,2^jk^{1/\beta})} |\nabla f|^q(x) m(x) \right)^{\frac1q}
\end{multline*}
where we used the doubling property \eqref{defDimension} in the third inequality, and the Poincar\'e inequality \ref{Pq} in the last one. After reinjecting the estimates on $T_j$ in \eqref{PkN*bd1}, we conclude that
\[ \left| [P^{k-1} \nabla^*\nabla f](y) \right|  \lesssim  k^{\frac{s}{\beta} - 1} \sum_{j\in \N^*}  \exp(-c' \lambda^j)  \left( \frac{1}{V(y,2^{j}k^{1/\beta})} \sum_{x\in B(y,2^jk^{1/\beta})} |\nabla f|^q(x) m(x) \right)^{\frac1q} .\]
The lemma follows. \ep

\subsection{Vertical Lusin functionals}

\begin{lemma}
Let $(G,\mu)$ be a graph satisfying \eqref{LB}, \eqref{D}, \ref{UE} for a $\beta \geq 2$, and \ref{Pq} for some $q\in [1,\infty]$ and $s \geq 1$. Then there exists $C>0$ such that for any for any collection $\{f_k\}_{k\in \N^*}$ of functions in $ L^p_{loc}(G)$ and any $x\in G$, 
\[ \mathcal C(\{k^{1-\frac{s}{\beta}} P^{k-1} \nabla^*\nabla f_k\}_{k\in \N^*})(x) \leq C \mathcal C^q(\{|\nabla f_k|\}_{k\in \N^*})(x).\]
\end{lemma}

\bp
Lemma \ref{lemLPq} immediately gives that 
\[  \mathcal C(\{k^{\frac{s}{\beta}-1} P^{k-1} \nabla^*\nabla f_k\}_{k\in \N^*})(x) \lesssim \sum_{j\in \N^*} e^{-c\lambda^j} \mathcal C^q_{(j)}(\{|\nabla f_k|\}_{k\in \N^*})(x).\]
With the input of Lemma \ref{lemLP2}, it becomes
\[  \mathcal C(\{k^{1-\frac{s}{\beta}} P^{k-1} \nabla^*\nabla f_k\}_{k\in \N^*})(x) \lesssim  \mathcal C^q(\{|\nabla f_k|\}_{k\in \N^*})(x) \sum_{j\in \N^*} e^{-c\lambda^j} 2^{jD} \lesssim  \mathcal C^q(\{|\nabla f_k|\}_{k\in \N^*})(x)\]
since $\sum_{j\in \N^*} e^{-c\lambda^j} 2^{jD} < +\infty$ is independent of $\{\omega_k\}_{k\in \N^*}$ or $x$. The lemma follows.
\ep

We continue with a bound on vertical Lusin functionals.

\begin{lemma} \label{lemLPc}
Let $(G,\mu)$ be a graph satisfying \eqref{LB}, \eqref{D}, \ref{UE} for a $\beta \geq 2$, and \ref{Pq} for some $q\in [1,\infty]$ and $s \geq 1$. Let $p\in [1,\infty)$.  Then there exists $C>0$ such that for any for any collection $\{f_k\}_{k\in \N^*}$ of functions in $ L^p_{loc}(G)$ and any $x\in G$, 
\[ \big\|\mathcal A(\{k^{1-\frac{s}{\beta}} P^{k-1} \nabla^*\nabla f_k\}_{k\in \N^*})\big\|_{L^p(G)} \leq C \big\| \mathcal A^q(\{|\nabla f_k|\}_{k\in \N^*})\big\|_{L^p(G)}.\]
\end{lemma}

\bp
Lemma \ref{lemLPq} immediately gives that 
\[  \big\|\mathcal A(\{k^{1-\frac{s}{\beta}} P^{k-1} \nabla^*\nabla f_k\}_{k\in \N^*})\big\|_{L^p(G)} \lesssim \sum_{j\in \N^*} e^{-c\lambda^j} \big\| \mathcal A^q_{(j)}(\{|\nabla f_k|\}_{k\in \N^*})\big\|_{L^p(G)}.\]
Lemma \ref{lemLP2A} yields then
\begin{multline*}
 \big\|\mathcal A(\{k^{1-\frac{s}{\beta}} P^{k-1} \nabla^*\nabla f_k\}_{k\in \N^*})\big\|_{L^p(G)}   \lesssim  \big\| \mathcal A^q(\{|\nabla f_k|\}_{k\in \N^*})\big\|_{L^p(G)} \sum_{j\in \N^*} e^{-c\lambda^j} 2^{jD} \\ \lesssim  \big\| \mathcal A^q(\{|\nabla f_k|\}_{k\in \N^*})\big\|_{L^p(G)}
 \end{multline*}
since $\sum_{j\in \N^*} e^{-c\lambda^j} 2^{jD} < +\infty$ is independent of $\{\omega_k\}_{k\in \N^*}$ or $x$. The lemma follows.
\ep

\subsection{Proof of Theorem \ref{ThMain} and Corollary \ref{CorMain}}.

\medskip

\noindent {\em Proof of Theorem \ref{ThMain}.} Let $\frac{s}{\beta} < \gamma < 1$, and $q<p<\infty$. 
Assume that $f$ is a compactly supported function on $G$. Then $f\in L^p(G)$, hence $\Delta^{\gamma} f \in L^p$ since $\Delta$ is a bounded operator. Theorem \ref{theoLa} implies that 
\[ \|\Delta^{\gamma} f\|_{L^p} \lesssim \big\|\mathcal A( \{k^{1-\gamma} \Delta P^{k-1} f\}_{k\in \N^*})\big\|_{L^p(G)} = \big\|\mathcal A( \{k^{1-\frac{s}{\beta}} P^{k-1} \nabla^* \nabla [ k^{\frac{s}{\beta} - \gamma}f]\}_{k\in \N^*})\big\|_{L^p(G)}.\]
Lemma \ref{lemLPc} now entails that
\[ \|\Delta^{\gamma} f\|_{L^p} \lesssim   \big\| \mathcal A^q(\{k^{\frac{s}{\beta}-\gamma}|\nabla f|\}_{k\in \N^*})\big\|_{L^p(G)}.\]
 However, if $\mathcal M$ is the uncentered Hardy-Littlewood maximal function,  
 \begin{multline*}
 \mathcal A^q(\{k^{\frac{s}{\beta}-\gamma}|\nabla f|\}_{k\in \N^*})(x) \\
  = \Big( \sum_{k\in \mathcal \N^*} k^{\frac{s}{\beta}-\gamma-1} \frac{1}{V(x,k)}   \sum_{y\in B(x,k^{1/\beta})} m(y) \Big[\underbrace{\frac{1}{V(y,k^{1/\beta})} \sum_{z\in B(y,k^{1/\beta})} |\nabla f|^q(z) m(z)}_{\leq  |\mathcal M(|\nabla f|^q)(x)} \Big]^{\frac{2}{q}} \Big)^\frac12 \\
  \leq   |\mathcal M(|\nabla f|^q)(x)|^\frac1q \sum_{k\in \N^*} k^{\frac{s}{\beta}-\gamma-1} \lesssim  |\mathcal M(|\nabla f|^q)(x),
 \end{multline*}
since $\frac{s}{\beta}-\gamma-1 < -1$, hence $\sum_{k\in \N^*} k^{\frac{s}{\beta}-\gamma-1} < \infty$. We deduce that 
\[  \|\Delta^{\gamma} f\|_{L^p}  \lesssim \big\| [\mathcal M(|\nabla f|^q)]^{\frac1q} \big\|_{L^p(G)} =  \big\|\mathcal M(|\nabla f|^q)\big\|_{L^{p/q}(G)}^{\frac1q} \lesssim \||\nabla f|^q\|_{L^{p/q}(G)}^{1/q} = \||\nabla f|\|_{L^{p}(G)}\]
by the $L^{p/q}$-boundedness of the maximal function $\mathcal M$, which holds since $p>q$. 
The theorem follows.
\ep

\medskip

Corollary \ref{ThMain} is an immediate consequence of Theorem \ref{CorMain} and the following lemma.

\begin{lemma} \label{lemP1}
Let $(G,\mu)$ be a graph satisfying \eqref{LB} and \eqref{D}. Let $D_G$ be an upper bound on the dimension as in \eqref{defDimension}. Then the Poincar\'e inequality (P$_1$)$_{D_G}$ holds. We even have the stronger inequality: there exists $C>0$ such that for any $u \in L^1(B(x,r))$, we have
\begin{equation} \label{sPq}
\sum_{z\in B(x,r)} |u(z) - u(x)| m(z) \leq C r^{D_G}  \sum_{y,z\in B(x,r) \atop y\sim z} |\nabla u(y,z) |\mu_{yz}
\end{equation}
\end{lemma}

\bp
Let $B = B(x,r) \subset G$ be a ball, and let $u$ a function on $G$. Then, since $B(x,r)$ is connected, 
\[ |u(z) - u(x)| \leq \sum_{y,y'\in B(x,r) \atop y'\sim y} |u(y) - u(y')|  = \sum_{y,y'\in B(x,r) \atop y'\sim y} |\nabla u(y,y')|.\]
But we also have, by \eqref{LB} and \eqref{defDimension}, that
\begin{multline*}
\sum_{y,y'\in B(x,r)\atop y\sim y'} |\nabla u(y,y')| =  \frac1{V(x,r)}  \sum_{y,y'\in B(x,r) \atop y\sim y'}|\nabla u(y,y')| \underbrace{m(y)}_{\leq \mu_{yy'}} \underbrace{\frac{V(x,r)}{m(y)}}_{\lesssim r^{D_G}} \\
\lesssim  \frac{r^{D_G}}{V(x,r)}  \sum_{y,y'\in B(x,r) \atop y\sim y'}|\nabla u(y,y')| \mu_{yy'} ,
\end{multline*}
hence
\begin{multline*}
\sum_{z\in B(x,r)} |u(z) - u(x)| m(z) \lesssim r^{D_G}  \frac1{V(x,r)}  \sum_{z\in B(x,r)} \sum_{y,y'\in B(x,r) \atop y \sim y'} |\nabla u(y,y')| \mu_{yy'} \\  = r^{D_G}  \sum_{y,y'\in B(x,r) \atop y \sim y'} |\nabla u(y,y')| \mu_{yy'},
\end{multline*}
which is \eqref{sPq}.

\medskip

 To prove the lemma, it remains to observe that (P$_q$)$_{1 + \frac{D-1}{q}}$ is a consequence of \eqref{sPq}. Indeed, first observe that by \eqref{sPq}
 \begin{multline} \label{sPq2}
 \sum_{z\in B(x,r)} |u(z) - u(x)| m(z) \lesssim \sum_{y\in B(x,r)} \sum_{y'\sim y} |\nabla u(y,y')| \mu_{yy'} \\ \leq \sum_{y\in B(x,r)} \left( \sum_{y'\sim y} |\nabla u(y,y')|^2 p(y',y) m(y') \right)^\frac12 m(y) = \sqrt 2 \sum_{y\in B(x,r)}  |\nabla u|(y) m(y).
 \end{multline}
 Then
 \begin{multline*}
 \|f-f_{B(x,r)}\|_{L^1(B(x,r))} \leq \|f-f(x)\|_{L^1(B(x,r))} + V(x,r) |f(x) - f_{B(x,r)}| \\
 \leq 2 \|f-f(x)\|_{L^1(B(x,r))} \lesssim r^{D_G} \||\nabla f|\|_{L^1(B(x,r))}
 \end{multline*}
 by \eqref{sPq2}. The lemma follows.
\ep

\section{Consequences for the Vicsek graphs}

\label{S5}

\begin{theorem} \label{theoInterp}
Let $(G,\mu)$ be a $D$-dimensional Vicsek graph, and let $S \subset G$ be a connected component. For $p\in [1,\infty]$, define the Sobolev space 
\[ \dot W^{1,p}(S):= \{ f \in L^p_{loc}(G,m), \, \nabla f \in L^p(E_S,\mu)\},\]
where $E_S:= \{(x,y) \in E_G, \, (x,y)\in S^2\}$, equipped with the semi-norm
\[ \|f\|_{\dot W^{1,p}(S)}:= \|\nabla f\|_{L^p(E_S,\mu)}.\]

Then the Sobolev spaces interpolate with the complex method. That is, if $X_1$ and $X_2$ are two compatible Banach spaces for complex interpolation, if $\theta \in [0,1]$ and $p,p_1,p_2 \in [1,\infty]$ verify $\frac1p = \frac{\theta}{p_1} + \frac{1-\theta}{p_2}$, and if $M :\, \dot W^{1,p_1}(S) \cup \dot W^{1,p_1}(S) \to X \cup Y$ is a maps verifying, for $i\in \{1,2\}$, 
\[ \|Mf\|_{X_i} \leq C_{i} \|f\|_{\dot W^{1,p_i}(S)} \qquad \text{ for } f \in \dot W^{1,p_i}(S).\] 
Then 
\begin{equation} \label{Interpolationccl}
\|Mf\|_{[X_1,X_2]_\theta} \leq 2 (C_1)^\theta (C_2)^{1-\theta} \|f\|_{\dot W^{1,p}(S)} \qquad \text{ for } f \in \dot W^{1,p}(S).
\end{equation}
\end{theorem}

\bp
Let $x \in S$ fixed randomly. Since $G$ is a tree, we can define a partial order on $G$ as follows: for $y,z\in G$, we write $y\prec z$ if $y\neq z$ and the (shortest) path from $x$ to $z$ goes through $y$. From there, we can see 1-forms on $S$ (as defined in Subsection \ref{defgraphs}) as functions on 
\[\bar E_S := \{(y,z)\in S\times S, \, y \sim z, \, y \prec z\} = \{(y,z) \in E_S, \ y \prec z\},\]
 and construct a primitive of the 1-form $\omega$ as a function on $S$ defined by
\[\Omega(z) := \sum_{i=1}^{\delta(x,z)} \omega(e_i),\]
where $e_i = (x_{i-1},x_i) \in \bar E_G$ satisfies  $x = x_0$, $z = x_{\delta(x,z)}$, and $x_{i-1} \sim x_i$ for all $i\in \{1,\dots, \delta(x,z)\}$. Indeed, the primitive is well defined since the shortest path from $x\in S$ to $z\in S$ necessary stays in $S$, because $S$ is connected inside a tree. 

\medskip

We define then $T$ be the operator
\[ T: \omega  \to \Omega,\]
and since $T\nabla f = f$ and $\nabla T\omega = \omega$, we have that
\[ T: \omega \in \bigcup_{q\in [1,\infty]} L^q(\bar E_S)  \to \Omega \in \bigcup_{q\in [1,\infty]} W^{1,q}(S)\]
with, for all $q\in [1,\infty]$
\[ \|T\omega\|_{\dot W^{1,q}(S)} = \|\omega\|_{L^q(\bar E_S,\mu)}.\]

\medskip

Let $X_1$, $X_2$, $p_1$, $p_2$, $\theta$, $p$, and $M$ like in the theorem. We have then, for $i\in \{1,2\}$
\[ \|MT \omega\|_{X_i} \leq C_1 \|T \omega\|_{\dot W^{1,p_i}(S)} = C_i \|\omega\|_{L^{p_i}(\bar E_S,\mu)} \qquad \text{ for } \omega \in L^{p_i}(\bar E_S,\mu),\]
so by complex interpolation, we get that 
\[ \|MT \omega\|_{[X_1,X_2]_\theta} \leq 2(C_1)^\theta(C_2)^{1-\theta} \|\omega\|_{L^{p}(\bar E_S,\mu)} \qquad \text{ for } \omega \in L^{p}(\bar E_S,\mu).\]
Applying the above inequality for $\omega = \nabla f$ and $f = T \omega$, we deduce \eqref{Interpolationccl}. The theorem follows.
\ep

\begin{corollary} \label{CorRRpInterp}
Let $(G,\mu)$ be a $D$-dimensional Vicsek graph. For $i\in \{1,2\}$, let $p_i \in (1,\infty)$ and $\gamma_1 \in [0,1]$ such that such {\normalfont (RR$_{p_i}$)$_{\gamma_i}$} holds. Then the reverse inequality \eqref{RRpggbis} is also valid whenever there exists $\theta \in [0,1]$ such that 
\begin{equation} \label{defpgamma}
 \frac1p = \frac\theta{p_1} + \frac{1-\theta}{p_2} \quad \text{ and } \quad \gamma = \theta \gamma_1 + (1-\theta)\gamma_2.
 \end{equation}
\end{corollary}

\bp
For $i\in \{1,2\}$, let $X_i$ be the Banach space 
\[X_i:= \left\{ \{f_k\}_{k\in \N^*}, \, \Big\| \Big( \sum_{k\in \N^*} k^{3-2\gamma_i} |f_k|^2 \Big)^\frac12\Big\|_{L^{p_i}(G)}  < + \infty\right\}\]
and then let  $M$ be the map 
\[ M(f) = \{\Delta^2 P^{k-1} f\}_{k\in \N^*}.\]

Theorem \ref{theoga}, (RR$_{p_i}$)$_{\gamma_i}$, and \eqref{LB} give that, for any $i\in \{1,2\}$,
\[\|M(f)\|_{X_i} = \|g_{2-\gamma_i} \Delta^{\gamma_i} f \|_{L^{p_i}(G)} \lesssim \|\Delta^{\gamma_i} f \|_{L^{p_i}(G)} \lesssim \||\nabla f|\|_{L^{p_i}(G)}  \lesssim \|\nabla f\|_{L^{p_i}(E_G)}\]
for  $f\in W^{1,p_i}(G)$. If $\theta \in [0,1]$, and $p$ and $\gamma$ are like in \eqref{defpgamma}, Theorem \ref{theoInterp} yields then that 
\[ \|M(f)\|_{[X_1,X_2]_\theta} \lesssim \|\nabla f\|_{L^{p}(E_G)} \lesssim  \||\nabla f|\|_{L^{p}(G)}  \qquad \text{ for } f\in W^{1,p}(G)\]
by \eqref{LB}. But since a classical result of complex interpolation (see for instance \cite[Theorems 5.6.3 and 5.1.2]{BerghL}) gives
\[[X_1,X_2]_\theta := \left\{ \{f_k\}_{k\in \N^*}, \, \Big\| \Big( \sum_{k\in \N^*} k^{3-2\gamma} |f_k|^2 \Big)^\frac12\Big\|_{L^{p_i}(G)}  < + \infty\right\},\]
we deduce, by using Theorem \ref{theoga}  again, that 
\[\|M(f)\|_{[X_1,X_2]_\theta} = \|g_{2-\gamma}(\Delta^\gamma f)\|_{L^p(G)} \gtrsim \|\Delta^\gamma f\|_{L^p(G)} \qquad \text{ for } f\in W^{1,p}(G).\]
Altogether, $\|\Delta^\gamma f\|_{p} \lesssim \|\nabla f\|_{p}$. The Corollary follows.
\ep

Next corollary is the Poincar\'e inequality on the Vicsek graphs. Analogue results for Vicsek fractals and Vicsek cable systems have been proved in \cite[Corollary 3.14]{BC23} and \cite[Lemma 2.6]{DR24} respectively.

\begin{corollary} \label{CorPqV}
Let $(G,\mu)$ be a $D$-dimensional Vicsek graph. Then, for any $q\in [1,\infty]$, the Poincar\'e inequality {\normalfont (P$_q$)$_{1 + \frac{D-1}{q}}$} holds. More precisely, there exists $C>0$ such that, for any $q\in [1,\infty]$, any $x\in G$, any $r>0$, and any function $u$ on $G$,
\begin{equation} \label{sPqV}
\left(\sum_{z\in B(x,r)} |u(z) - u(x)|^q m(z)\right)^\frac1q \leq C r^{1 + \frac{D-1}{q}}  \left(  \sum_{z,z'\in B(x,r) \atop z\sim z'} |\nabla u(z,z')|^q \mu_{zz'} \right)^\frac1q.
\end{equation}
\end{corollary}

\bp
Let $x\in G$ and $r>0$. Define the map $M$ for functions on $B(x,r)$ as
\[ M(f) = f - f(x).\]
With the notation of Theorem \ref{theoInterp}, Lemma \ref{lemP1} gives that 
\[ \|M(f)\|_{L^1(B(x,r))} \lesssim r^{D}  \|f\|_{\dot W^{1,1}(B(x,r))}.\]
Moreover, since any point of $B(x,r)$ is connected to $x$ by a path staying in $B(x,r)$ of length at most $r$ - this comes directly from the definition of the metric $\delta$ and the balls - we have
\[ \|M(f)\|_{L^\infty(B(x,r))} \leq r \sup_{z,z'\in B(x,r) \atop z\sim z'} |f(z) - f(z')| = r \|f\|_{\dot W^{1,\infty}(B(x,r))}.\]
Theorem \ref{theoInterp} gives that, for any $q\in [1,\infty]$, 
\[ \|M(f)\|_{L^q(B(x,r))} \lesssim r^{1 + \frac{D-1}{q}} \|f\|_{\dot W^{1,q}(B(x,r))},\]
which is exactly \eqref{sPqV}.

  To prove the theorem, it suffices to observe that (P$_q$)$_{1 + \frac{D-1}{q}}$ is a consequence of \eqref{sPqV}. Indeed,
 \begin{multline*}
 \|f-f_{B(x,r)}\|_{L^q(B(x,r))} \leq \|f-f(x)\|_{L^q(B(x,r))} + V(x,r) |f(x) - f_{B(x,r)}| \\
 \leq 2 \|f-f(x)\|_{L^q(B(x,r))} \lesssim r^{1 + \frac{D-1}{q}} \|\nabla f\|_{L^q(E_{B(x,r)})} \lesssim r^{1 + \frac{D-1}{q}} \||\nabla f|\|_{L^q(B(x,r))}
 \end{multline*}
 by \eqref{sPqV} and then \eqref{LB}. The corollary follows.
\ep

\medskip

We conclude the article with the proof of Theorem \ref{ThVicsek}.

\medskip

\noindent {\em Proof of Theorem \ref{ThVicsek}.} We need to prove the following 4 facts:
\begin{enumerate}[(a)]
\item \eqref{RRpggbis} holds when  $p\in (1,\infty$) and $\gamma^*(p) < \gamma < 1$;
\item \eqref{RRpggbis} fails when  $p\in (1,\infty$) and $0 < \gamma < \gamma^*(p) $;
\item \eqref{Rpggbis} fails when  $p\in (1,\infty$) and $\gamma^*(p) < \gamma < 1$;
\item \eqref{Rpggbis} holds when  $p\in (1,\infty$) and $0 < \gamma < \gamma^*(p) $;
\end{enumerate}
Let us show the 4 assertions, in order.

\medskip

\noindent {\bf Assertion (a).} By Corollary \ref{CorPqV}, for any $q\in [1,\infty]$, the Poincar\'e inequality (P$_q$)$_{1+\frac{D-1}{q}}$ holds. Together with Theorem \ref{ThMain}, it means that  \eqref{RRpggbis} whenever there exists $q \in [1,\infty]$ satisfying $q<p$ and 
\[  \frac{1+\frac{D-1}{q}}{D+1} < \gamma < 1,\]
i.e. whenever 
\[\gamma > \frac{1+\frac{D-1}{p}}{D+1} = \frac1{D+1} + \frac1p \Big( \frac{D-1}{D+1} \Big) = \gamma^*(p).\]

\medskip

\noindent {\bf Assertion (b).} Assume by contradiction that there exists $p_0\in (1,\infty)$ and $0 < \gamma_0 < \gamma^*(p_0)$ such that (RR$_{p_0,\gamma_0}$) holds. Assertion (a) gives us also that  \eqref{RRpggbis} holds whenever $p\in (1,\infty$) and $\gamma^*(p) < \gamma < 1$. By interpoling those two situations - using Corollary \ref{CorRRpInterp} - we deduce that (RR$_{2,\gamma}$) holds for a $\gamma < 2$. However, by \cite[Lemma 4.4]{Fen26}, the later is impossible, hence the contradiction.

\medskip

\noindent {\bf Assertion (c).} This comes directly from Assertion (b) and the duality given in Lemma \ref{lemDuality}.

\medskip

\noindent {\bf Assertion (d).} 
First, observe that $\||\nabla f|\|_{L^p(G)} \approx \|\nabla f\|_{L^p(E_G)}$. Fix $x\in G$ and define the set $E_G$ and the map $T$ as in the proof of Theorem \ref{theoInterp}. We have $ \|\nabla f\|_{L^p(E_G)}^p =  2 \|\nabla f\|_{L^p(\bar E_G)}^p$, and thus 
\[\||\nabla f|\|_{L^p(G)} \lesssim \sup_{\omega \in L^{p'}(\bar E_G) \atop \|\omega\|_{p'} \leq 1} \sum_{e\in \bar E_G} \nabla f(e) \, \omega(e) \, \mu(e) =  \sup_{\omega \in L^{p'}(\bar E_G) \atop \|\omega\|_{p'} \leq 1} \sum_{e\in \bar E_G} \nabla f(e) \, \nabla T\omega(e) \, \mu(e)
\]
We use the fact that $E_G = \{(y,z) \in G^2, \, (y,z) \text{ or } (z,y) \in \bar E_G\}$, and that $\nabla u(y,z) = - \nabla u(z,y)$ for $(y,z)\in E_G$ to say that 
\begin{multline} \label{trucA}
\||\nabla f|\|_{L^p(G)} \lesssim \sup_{\omega \in L^{p'}(\bar E_G) \atop \|\omega\|_{p'} \leq 1} \sum_{e\in E_G} \nabla f(e) \, \nabla T \omega(e) \, \mu(e) \\ = \sup_{\omega \in L^{p'}(\bar E_G) \atop \|\omega\|_{p'} \leq 1} \sum_{z\in G} \Delta f(z) \, T \omega(z) \, m(z)  = \sup_{\omega \in L^{p'}(\bar E_G) \atop \|\omega\|_{p'} \leq 1} \sum_{z\in G} \Delta^\gamma f(z) \, [\Delta^{1-\gamma} T \omega](z) \, m(z) \\
\leq  \|\Delta^\gamma f\|_{L^p(G)} \sup_{\omega \in L^{p'}(\bar E_G) \atop \|\omega\|_{p'} \leq 1} \, \|\Delta^{1-\gamma} T \omega\|_{L^{p'}(G)}
\end{multline}
where the first equality of the second line is due to \eqref{dd*duals}. But with our assumptions on $\gamma$ and $p$, we have $\gamma^*(p') < 1-\gamma < 1$, meaning that we can invoke Assertion (a) to obtain
\[ \|\Delta^{1-\gamma} T \omega\|_{L^{p'}(G)} \leq \||\nabla  T \omega| \|_{L^{p'}(G)} \approx \|\nabla  T \omega \|_{L^{p'}(\bar E_G)}\]
by \eqref{LB}. But since $\nabla T$ is the identity on $E_G$, we deduce that 
\[\|\Delta^{1-\gamma} T \omega\|_{L^{p'}(G)} \lesssim \|\omega\|_{L^{p'}(\bar E_G)} \]
and then, by reinjecting this last bound in \eqref{trucA}, 
\[\||\nabla f|\|_{L^p(G)} \lesssim  \|\Delta^\gamma f\|_{L^p(G)}.\]
Assertion (d) and then the Theorem follows.
\ep

\bibliographystyle{amsalpha}

\end{document}